\documentclass[a4paper, 10pt, english]{article}
\usepackage[utf8]{inputenc}
\usepackage[T1]{fontenc}
\usepackage{graphicx}
\usepackage{stmaryrd}
\usepackage[a4paper]{geometry}
\usepackage{amsmath,amsfonts,amssymb,amsthm,epsfig,epstopdf,url,array}
\usepackage{rotating}
\usepackage[colorlinks=true,citecolor=red,linkcolor=blue,pdfpagetransition=Blinds]{hyperref}
\usepackage{cleveref}
\usepackage{nameref}
\usepackage[inline]{enumitem}
\usepackage{comment}
\Crefname{paragraph}{Section}{Sections}
\usepackage{fancyhdr}
\usepackage{fullpage}
\usepackage{mathrsfs}
\usepackage{appendix}

\numberwithin{equation}{section}

\crefname{lem}{lemma}{lemmas}
\crefname{theo}{theorem}{theorems}

\providecommand{\keywords}[1]{\noindent {\textit{Keywords:}} #1}

\usepackage{color}
\newcommand{\R}{{\mathbb{R}}}

\renewcommand{\O}{{\mathcal{O}}}
\theoremstyle{plain} 
\newtheorem{prop}{Proposition}[section] 
\newtheorem{theo}[prop]{Theorem}

\newtheorem{lem}[prop]{Lemma}

\theoremstyle{definition}

\newtheorem{rmk}[prop]{Remark}

\newcommand{\norme}[1]{\left\lVert#1\right\rVert}
\newcommand{\snorme}[1]{\lVert #1 \rVert}
\newcommand\dt{\triangle t}
\newcommand\sdt{\scriptscriptstyle \triangle t}

\newcommand\mT{\mathcal P}
\newcommand\dmT{\mathcal D}
\newcommand\smT{{\scriptscriptstyle\mT}}
\newcommand\sdmT{{\scriptscriptstyle\dmT}}
\newcommand{\ov}[1]{\overline{#1}}
\newcommand{\difx}{\partial_x}

\newcommand{\sht}{\mathtt t}
\newcommand{\taup}[1]{(\mathtt{t}^+#1)}

\newcommand{\taum}[1]{(\mathtt{t}^-#1)}

\newcommand{\taubp}[1]{(\bar{\mathtt t}^+#1)}
\newcommand{\taubm}[1]{(\bar{\mathtt t}^-#1)}
\newcommand{\tbp}[1]{\bar{\mathtt t}^+#1}
\newcommand{\tbm}[1]{\bar{\mathtt t}^-#1}
\newcommand{\tcp}[1]{{\normalfont\textsf t}^+#1}
\newcommand{\tcm}[1]{{\normalfont\textsf t}^-#1}
\newcommand\Dtbar{\overline{{D}}_t}
\newcommand\Dt{D_t}
\newcommand\Dtc{\normalfont\textsf{D}_t}
\newcommand\ddbint{\,\dint\!\!\!\int}
\newcommand\dbint{\int\!\!\!\int}
\newcommand\D{\displaystyle}
\def\dt{{\triangle t}}

\makeatletter
\newcommand{\inlineitem}[1][]{%
\ifnum\enit@type=\tw@
    {\descriptionlabel{#1}}
  \hspace{\labelsep}%
\else
  \ifnum\enit@type=\z@
       \refstepcounter{\@listctr}\fi
    \quad\@itemlabel\hspace{\labelsep}%
\fi}
\makeatother

\def\Yint#1{\mathchoice
    {\YYint\displaystyle\textstyle{#1}}%
    {\YYint\textstyle\scriptstyle{#1}}%
    {\YYint\scriptstyle\scriptscriptstyle{#1}}%
    {\YYint\scriptscriptstyle\scriptscriptstyle{#1}}%
      \!\int}
\def\YYint#1#2#3{{\setbox0=\hbox{$#1{#2#3}{\int}$}
    \vcenter{\hbox{$#2#3$}}\kern-.52\wd0}}

\def\dint{\Yint{\textrm{---}}}

\newcommand\inter[1]{\llbracket #1\rrbracket}

 \usepackage{tikz}
 \usepackage{pgfplots}
 \usepackage{subfig}
 \pgfplotsset{surface/.style={ %
               xmax=1.1,%
               axis z line=center,%
               axis x line=center,%
               axis y line=center,%
               zmin=0,%
               clip=false,%
               extra x ticks={1},%
               extra x tick label={$T=1$},%
               xtick={0},%
               ytick=\empty}}

 \pgfplotsset{erreurs/.style={scale=1,
     legend cell align=left,
     legend pos=outer north east,
     legend plot pos=right,
     legend style={cells={anchor=east},draw=none},
     xlabel=$h$,
    xmin=0.001,xmax=0.03}}

\tikzset{pente/.style={opacity=0.6}}

\pgfplotsset{cout/.style={black,mark=diamond*,mark size=2.5,mark options={fill=gray}}}
\pgfplotsset{cible/.style={black,mark=square*,mark size=2.5,mark options={fill=gray}}}
\pgfplotsset{cibleyT/.style={black,mark=*,mark size=2.5,mark options={fill=gray}}}
\pgfplotsset{CG/.style={black,mark=otimes*,mark size=2.5,mark options={fill=gray}}}
\pgfplotsset{solex/.style={black,mark=*,mark size=2.5,mark options={fill=gray}}}
\pgfplotsset{energie/.style={black,mark=triangle*,mark size=2.5,mark options={fill=gray}}}

\usetikzlibrary{matrix,external}
\usetikzlibrary{shapes,backgrounds}
\usetikzlibrary{shapes.misc}
\tikzset{cross/.style={cross out, draw=blue, fill=none, minimum size=2*(#1-\pgflinewidth), inner sep=0pt, outer sep=0pt}, cross/.default={4pt}}

\makeatletter
\let\original@addcontentsline\addcontentsline
\newcommand{\dummy@addcontentsline}[3]{}
\newcommand{\DeactivateToc}{\let\addcontentsline\dummy@addcontentsline}
\newcommand{\ActivateToc}{\let\addcontentsline\original@addcontentsline}
\makeatother

\begin{document}

\title{\bf Controllability of a simplified time-discrete stabilized Kuramoto-Sivashinsky system}

\author{V\'ictor Hern\'andez-Santamar\'ia\thanks{Instituto de Matem\'aticas, Universidad Nacional Aut\'onoma de M\'exico, Circuito Exterior, C.U., C.P. 04510 CDMX, Mexico. E-mail: \texttt{victor.santamaria@im.unam.mx}}} 
\maketitle

\begin{abstract}
In this paper, we study some controllability and observability properties for a coupled system of time-discrete fourth- and second-order parabolic equations. This system can be regarded as a simplification of the well-known stabilized Kumamoto-Sivashinsky equation. Unlike the continuous case, we can prove only a relaxed observability inequality which yields a $\phi(\dt)$-controllability result. This result tells that we cannot reach exactly zero but rather a small target whose size goes to 0 as the discretization parameter $\dt$ goes to 0. The proof relies on a known Carleman estimate for second-order time-discrete parabolic operators and a new Carleman estimate for the time-discrete fourth-order equation.
\end{abstract}

%\footnotesize
%\tableofcontents
%\normalsize

\section{Introduction}

The stabilized Kuramoto-Sivashinsky (SKS) equation was proposed in \cite{FMK03} as a model of front propagation in reaction-diffusion phenomena and combines dissipative features with dispersive ones. This systems consists of a Kuramoto-Sivashinsky-KdV (KS-KdV) equation linearly coupled to an extra dissipative equation. This model takes the form
\begin{equation}\label{ks_intro}
\begin{cases}
u_t-\Gamma u_{xx}+u_{x}=v_{x}, \\
v_t+\gamma v_{xxxx}+v_{xxx}+av_{xx}+uu_{x}=u_x, 
\end{cases}
\end{equation}
where the coefficients $\gamma,a>0$ take into account the long-wave instability and the short wave dissipation, respectively, $\Gamma>0$ is the dissipative parameter and $c\in\mathbb R\setminus\{0\}$ is the group velocity mismatch between wave modes.  This model applies to the description of surface waves on multilayered liquid films and serves
as a one-dimensional model for turbulence and wave propagation, see \cite{FMK03} for a more detailed
discussion.

System \eqref{ks_intro} has been studied from the controllability point of view in the recent past. In \cite{CMP12}, the authors addressed for the first time the null-controllability of the SKS model when the control action is applied at the boundary of both equations. Later, in \cite{CMP15}, it has been proved the null-controllability of the coupled system with a single control applied at the interior of the fourth-order equation, while in \cite{CC16} the analogous result is proved by controlling from the heat equation. Stochastic versions for some of these results can be found in \cite{HSP20} and some boundary control problems for simplified systems have been studied in \cite{CCM19} and \cite{HMV20}.

In this paper, our main interest is to study some controllability and observability properties for time-discrete approximations of a simplified SKS model. As it is well-known, controllability/observability and numerical discretization do not commute well (see e.g. \cite{Zua05}) but at least we expect to retain some properties.

\subsection{Problem formulation and main result}

In this paper, we shall use the notation $\inter{a,b}=[a,b]\cap \mathbb N$ for any real numbers $a<b$.

We are interested in studying the controllability of a simplified time-discrete version of \eqref{ks_intro}. To this end, let $M\in\mathbb N^*$ and set $\dt=T/M$. We introduce the following uniform discretization of the time variable
\begin{equation}\label{eq:discr_points}
0=t_0<t_1<\ldots<t_M=T,
\end{equation}
with $t_n=n\dt$ and $n\in\inter{0,M}$. We also introduce the midpoints $t_{n+\frac12}=(t_{n+1}+t_n)/2$ for $n\in\inter{0,M-1}$, see \Cref{fig:discr_T}.

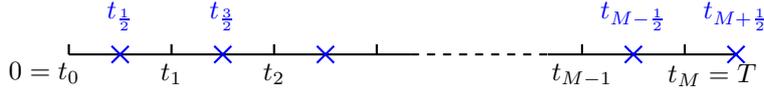
\begin{figure}[t]
\raggedright\hspace{1.75 cm}
\begin{tikzpicture}[scale=0.9]
\draw[-,thick] (-4,0) -- (1,0);
\draw[dashed,thick] (1, 0) -- (3, 0);
\draw[-,thick] (3,0) -- (5.75,0);
%marcas primal
%\draw [thick] (-4,0) node[below]{$t_0$} -- (-4,0.15);
\draw [thick] (-4,0) -- (-4,0.15);
\draw[thick,color=black] (-4.36,0.025) node[below]{$0=t_0$} ;
%\draw [thick] (-4,0) node[above left]{$0$} -- (-4,0);
\draw [thick] (-2.5,0) node[below]{$t_1$} -- (-2.5,0.15);
\draw [thick] (-1.0,0) node[below]{$t_2$} -- (-1.0,0.15);
\draw [thick] (0.5,0) node[below]{} -- (0.5,0.15);
\draw [thick] (3.5,0) node[below]{$t_{M-1}$} -- (3.5,0.15);
%\draw [thick] (5,0) node[below]{$t_{M}$} -- (5,0.15);
\draw [thick] (5,0) -- (5,0.15);
\draw[thick,color=black] (5.4,0.005) node[below]{$t_M=T$} ;
%\draw [thick] (5,0) node[above right]{$T$} -- (5,0);
%marcas dual
\draw[thick] (-3.25,0) node[cross]{} -- (-3.25,0);
\draw[thick,color=blue] (-3.25,0.2) node[above]{$t_{\frac{1}{2}}$} ;
\draw[thick] (-1.75,0) node[cross]{} -- (-1.75,0);
\draw[thick,color=blue] (-1.75,0.2) node[above]{$t_{\frac{3}{2}}$} ;
\draw[thick] (-0.25,0) node[cross]{} -- (-0.25,0);
\draw[thick,color=blue] (4.25,0.2) node[above]{$t_{M-\frac{1}{2}}$} ;
\draw[thick] (4.25,0) node[cross]{} -- (4.25,0);
\draw[thick,color=blue] (5.75,0.2) node[above]{$t_{M+\frac{1}{2}}$} ;
\draw[thick] (5.75,0) node[cross]{} -- (5.75,0);
\put(175,10){$\overline{\mT}=(t_n)_{n\in\inter{0,M}}$}
\put(175,-10){$\textcolor{blue}{\overline{\dmT}=(t_{n+\frac{1}{2}})_{n\in \inter{0, M}}}$}
\end{tikzpicture}
\caption{Discretization of the time variable and its notation.}
\label{fig:discr_T}
\end{figure}

For any time-discrete control sequence $h=\{h^{n+\frac12}\}_{n\in\inter{0,M-1}}\subset L^2(\Omega)$, consider the sequence $(u,v)=\{u^n,v^n\}_{n\in\inter{0,M}}\subset L^2(\Omega)$ verifying the recursive formula 
\begin{equation}\label{eq:ks_heat}
\begin{cases}
\D \frac{u^{n+1}-u^{n}}{\dt}-\Gamma \difx^2 u^{n+1}+c\difx u^{n+1}={v^{n+1}} &n\in\inter{0,M-1}, \\
\D \frac{v^{n+1}-v^{n}}{\dt}+\gamma \difx^4 v^{n+1}+\difx^3 v^{n+1}+a\difx^2 v^{n+1}={ u^{n+1}}+\chi_{\omega} h^{n+\frac12} &n\in\inter{0,M-1}, \\
\left(v_{|\partial\Omega}\right)^{n+1}=\left(v_{|\partial\Omega}\right)^{n+1}=\left(\difx v_{|\partial\Omega}\right)^{n+1} =0 &n\in\inter{0,M-1}, \\
u^0=u_0, \quad v^0=v_0,
\end{cases}
\end{equation}
where $(u^n,v^n)$ (resp. $h^{n+\frac12}$) denotes an approximation of $(u,v)$ (resp. $h$) at time $t_n$ (resp. $t_{n+\frac12}$). In fact, \eqref{eq:ks_heat} is an implicit Euler discretization of a simplified model for \eqref{ks_intro}. The main simplifications are that system \eqref{eq:ks_heat} is linear and that the couplings on the right-hand side are through zero order couplings (see \Cref{sec:conc} for further remarks about this).  

As in the continuous case, we can state a notion of null-controllability, that is, system \eqref{eq:ks_heat} is said to be null-controllable for any $(u_0,v_0)\in[L^2(\Omega)]^2$ if there exists a sequence $h=\{h^{n+\frac12}\}_{n\in\inter{0,M-1}}$ such that the corresponding satisfies
\begin{equation}\label{eq:uM_vM_0}
u^{M}=v^{M}=0.
\end{equation}
However, it turns out that \eqref{eq:uM_vM_0} is a very strong condition and in general cannot hold. In fact, as pointed out in \cite{zhe08}, even for the simple case of the heat equation, time-discretization does not preserve null-controllability and not even approximate one. 

For this reason, in this paper we shall pursue a different notion that captures some controllability properties and allows us to recover the known results in the continuous case. Our main result is the following.
 
\begin{theo}\label{thm:main_1}
Let $T\in(0,1)$ and a discretization parameter $\dt>0$ small enough. Then, for any $(u_0,v_0)\in [L^2(\Omega)]^2$ and any function $\phi$ verifying 
\begin{equation}\label{eq:cond_dt}
\liminf_{\dt\to 0} \frac{\phi(\dt)}{e^{-C_1/(\dt)^{1/10}}}>0,
\end{equation}
for some $C_1>0$ uniform with respect to $\dt$, there exists a time discrete control $v$ such that 
\begin{equation}\label{eq:cost_control_full}
\left(\sum_{n=0}^{M-1}\dt \int_{\omega} |v^{n+\frac12}|^2\right)^{1/2} \leq C \norme{(u_0,y_0)}_{[L^2(\Omega)]^2}
\end{equation}
and such that the associated solution $(u,v)$ to \eqref{eq:ks_heat} verifies
\begin{equation}\label{eq:phi_dt_contr}
\snorme{(u^{M},v^{M})}_{[L^2(\Omega)]^2}\leq C\sqrt{\phi(\dt)}\norme{(u_0,y_0)}_{[L^2(\Omega)]^2},
\end{equation}
where $C>0$ only depends on $\phi$, $T$, $\Omega$, $\omega$, $\Gamma$, $\gamma$, $a$, and $c$.
\end{theo}

Condition \eqref{eq:phi_dt_contr} is the so-called $\phi(\dt)$-controllability. This definition, introduced in \cite{BHLR10ar} (see also \cite{Boy13,BLR14}) for studying semi-discrete (i.e., in space) approximations of parabolic control problems, tell us that we do not steer the solution exactly zero but rather we reach a small target whose size goes to zero at the rate $\sqrt{\phi(\dt)}$. In fact, from \eqref{eq:cost_control_full} and \eqref{eq:phi_dt_contr}, we can recover (up to a subsequence) the expected result in the continuous case. 

The strategy we shall employ for controlling \eqref{eq:ks_heat} relies on time-discrete Carleman estimates and some well-known methods adapted to the discrete case. In particular, we follow the strategy outlined in \cite{CMP15} (see also \cite{GBdT10}) in which suitable Carleman estimates with the same weight are applied to each equation and then, by local energy estimates, one of the observations is removed. 

We shall mention that the time-discrete Carleman estimate for the second-order equation has been recently obtained in \cite{BHS20}, but the time-discrete Carleman estimate for the fourth-order operator was not available in the literature. In this regard, our second main contribution is \Cref{thm:time_discrete_fourth} which states precisely this estimate. Looking at such result, one can readily see the classical structure of the Carleman estimate for fourth-order parabolic equations (see \cite{Zho12,CMP15} for the 1-d case and \cite{GK19} for the multi-d one) but, as for the time-discrete heat equation (see \Cref{lem:carleman_heat}), several remainders appear. This is exactly why we can prove only \eqref{eq:phi_dt_contr}.

We would like to finish this introduction by mentioning that there are other approaches for controlling time-discrete systems, see for instance, \cite{zhe08,EZZ08,ZZZ09,EV10,BHLR11}, but they mainly rely on spectral techniques. Although some of these techniques have been used for controlling coupled systems of fourth- and second-order parabolic equations  at the continuous level (see e.g. \cite{Liu14,HSP20}), their application for controlling time-discrete coupled systems is yet to be studied. 

\subsection{Outline}

The rest of the paper is organized as follows. In \Cref{sec:not_time}, we provide some useful notations and definitions that simplify greatly the presentation. With these, our results and proofs will resemble closely those at the continuous level. \Cref{sec:relaxed_obs} is devoted to prove a Carleman estimate for the adjoint system \eqref{eq:adj_ks_heat} with only one observation which in turn yields a \textit{relaxed} observability inequality (see eq. \eqref{eq:obs_final}). This result is then employed in \Cref{sec:phi_dt_contr} to obtain the proof of \Cref{thm:main_1}. We present some final comments in \Cref{sec:conc}

\section{Some useful notation for the time-discrete scheme}\label{sec:not_time}

We devote this section to introduce some notations and definitions that allows us to use a formalism as close as possible to the continuous case. In this way, our results and computations can be followed and carried out in a very intuitive manner. 

From the discretization points \eqref{eq:discr_points}, we denote by $\mT:=\left\{t_n : n\in\inter{1,M}\right\}$ the (primal) set of points excluding the first one and we write $\ov{\mT}:=\mT\cup\{t_0\}$. To handle the approximation of time derivatives, we will work with the (dual) points $t_{n+\frac{1}{2}}$. Its collection is defined as $\dmT:=\{t_{n+\frac{1}{2}}:n\in\inter{0,M-1}\}$. It will be convenient to consider also an extra point $\{t_{M+\frac{1}{2}}\}$ which lies outside the time interval $[0,T]$ (see \Cref{fig:discr_T}) and to write $\ov{\dmT}:= \mT\cup\{T_{M+\frac{1}{2}}\}$. Observe that both $\mT$ and $\dmT$ have a total number of $M$ points. 

We denote by $\R^{\smT}$ and $\R^{\sdmT}$ the sets of real-valued discrete functions defined on $\mT$ and $\dmT$. If $u^{\smT}\in\R^{\smT}$ (resp. $u^\sdmT\in\R^{\sdmT}$), we denote by $u^n$ (resp. $u^{n+\frac{1}{2}}$) its value corresponding to $t_n$ (resp. $t_{n+\frac{1}{2}}$). For $u\in \R^{\smT}$ we define the time-discrete integral
\begin{equation}\label{eq:int_time_primal}
\int_{0}^{T}u^{\smT}:=\sum_{n=1}^{M}\dt\, u^n,
\end{equation}
and for $u^{\sdmT}\in \R^{\sdmT}$ we define 
\begin{equation}\label{eq:int_time_dual}
\dint_{0}^{T}u^{\sdmT}:=\sum_{n=0}^{M-1}\dt\, u^{n+\frac{1}{2}}.
\end{equation}
\begin{rmk}
To ease the notation and thanks to the introduction of two different integral symbols, in what
follows we shall write $u$ indistinctly to refer to functions $u^\smT$ or $u^{\sdmT}$. 
\end{rmk}

Let $\{X,|\cdot|_{X}\}$ be a real Banach space. We denote by $X^{\smT}$ and $X^{\sdmT}$ the sets of vector-valued functions defined on $\mT$ and $\dmT$, respectively. Using definitions \eqref{eq:int_time_primal} and \eqref{eq:int_time_dual}, we denote by $L_{\smT}^p(0,T;X)$ (resp. $L^p_{\sdmT}(0,T;X)$), $1\leq p<\infty$, the space $X^\smT$ (resp. $X^{\sdmT}$) endowed with the norm
\begin{equation*}
\norme{u}_{L^p_{\smT}(0,T;X)}:=\left(\int_{0}^{T}|u|^p_{X}\right)^{1/p} \quad \left(\textnormal{resp.} \quad \norme{u}_{L^p_{\sdmT}(0,T;X)}:=\left(\dint_{0}^{T}|u|^p_{X}\right)^{1/p}\right).
\end{equation*}
%

%We also define the space $L^\infty_{\smT}(0,T;X)$ (resp. $L^\infty_{\sdmT}(0,T;X)$) by means of the norm
%%
%\begin{equation}\label{eq:def_Linfty_time}
%\|u\|_{L^\infty_{\smT}(0,T;X)}:=\sup_{n\in\inter{1,M}} |u^n|_{X} \quad \left(\text{resp.}\quad  \|u\|_{L^\infty_{\sdmT}(0,T;X)}:=\sup_{n\in\inter{0,M-1}} |u^{n+\frac12}|_{X}  \right).
%\end{equation}

In the case where $p=2$ and $X$ is replaced by a Hilbert space $\{H,(\cdot,\cdot)_{H}\}$, $H^{\smT}$ (resp. $H^{\sdmT}$) becomes a Hilbert space for the norm induced by the inner product
\begin{equation}\label{eq:inner_prod}
\int_{0}^{T}\left(u,v\right)_{H}:=\sum_{n=1}^{M}\dt\,(u^n,v^n)_{H}  \quad \left(\text{resp.}\quad \dint_{0}^{T}\left(u,v\right)_{H}:=\sum_{n=0}^{M-1}\dt\,(u^{n+\frac12},v^{n+\frac12})_{H}\right).
\end{equation}
\begin{rmk}
Particularly, if $H=L^2(\Omega)$ we shall use the notation
\begin{equation*}
\dbint_{Q}uv:=\int_{0}^{T}(u,v)_{L^2(\Omega)} \quad \left(\text{resp.} \quad \ddbint_{Q}uv:=\dint_{0}^{T}(u,v)_{L^2(\Omega)}\right).
\end{equation*}
Also, in accordance with the notation used in the continuous case, we denote the space $L^2_{\mT}(0,T;L^2(\Omega))$ as $L^2_{\smT}(Q)$. Similar notation holds for functions defined on the dual grid $\dmT$. 
\end{rmk}

To manipulate time-discrete functions, we define translation operators $\sht^+:X^\smT \to X^{\sdmT}$ and $\sht^-:X^{\overline\smT} \to X^{\sdmT}$ as
\begin{equation*}%\label{def_taus_primal}
\taup{u}^{n+\frac12}:=u^{n+1}, \quad \taum{u}^{n+\frac 12}:=u^n, \quad n\in \inter{0,M-1}.
\end{equation*}
With this, we can define a difference operator $\Dt$ as the map from $X^{\overline\smT}$ into $X^\sdmT$ given by
\begin{equation*}%\label{dev_primal}
\begin{split}
&(\Dt u)^{n+\frac12}:= \frac{u^{n+1}-u^n}{\dt}=\left(\frac{1}{\dt}\left(\sht^+ - \sht^-\right)u\right)^{n+\frac12}, \quad n\in \inter{0,M-1}. \\
\end{split}
\end{equation*}

In the same manner, we can define the translation operators  $\bar{\sht}^+:X^{\overline\sdmT} \to X^{\smT}$ and $\bar{\sht}^-:X^{\sdmT} \to X^{\smT}$ as follows
\begin{equation*}%\label{def_taus_dual}
\taubp u^{n}:=u^{n+\frac 12}, \quad \taubm u^{n}=u^{n-\frac 12}, \qquad n\in \inter{1,M}, \\
\end{equation*}
as well as a difference operator $\Dtbar$ (mapping $X^{\overline{\sdmT}}$ into $X^{\smT}$) given by 
\begin{equation*}%\label{dev_dual}
(\Dtbar u)^{n}:= \frac{u^{n+\frac 12}-u^{n-\frac12}}{\dt}=\left(\frac{1}{\dt}\left(\bar{\mathtt t}^+-\bar{\mathtt t}^{-}\right)u\right)^{n}, \quad n\in \inter{1,M}.
\end{equation*}

These definitions, together with the integral symbols \eqref{eq:int_time_primal} and \eqref{eq:int_time_dual}, help us to write formulas and systems in a more compact way. For instance, \eqref{eq:cost_control_full} can be rewritten as $\norme{v}_{L^2_{\sdmT}(0,T;L^2(\omega))}\leq C \norme{(u_0,y_0)}_{[L^2(\Omega)]^2}$ while system \eqref{eq:ks_heat} can be written compactly as
\begin{equation*}
\begin{cases}
\D (\Dt u)^{n+\frac12}-\Gamma \difx^2 \taup{u}^{n+\frac12}+c\difx \taup{u}^{n+\frac12}=\taup{v}^{n+\frac12} &n\in\inter{0,M-1}, \\
\D (\Dt v)^{n+\frac12}+\gamma \difx^4 \taup{v}^{n+\frac12}+\difx^3 \taup{v}^{n+\frac12}+a\difx^2 \taup{v}^{n+\frac12}=\taup{u}^{n+\frac12}+\chi_{\omega} \widehat{q}^{\,n+\frac12} &n\in\inter{0,M-1}, \\
%\left(v_{|\partial\Omega}\right)^{n+1}=\left(v_{|\partial\Omega}\right)^{n+1}=\left(\difx v_{|\partial\Omega}\right)^{n+1} =0 &n\in\inter{0,M-1}, \\
u^0=u_0, \quad v^0=v_0,
\end{cases}
\end{equation*}
where we can omit the boundary conditions since no other conditions will be used in this document. 

Also, the above definitions allow us to handle the evaluation of continuous functions on the primal and dual meshes $\mT$ and $\dmT$ and enable to present an integration-by-parts (in time) formula in a very natural way. For instance, in the case of \eqref{eq:inner_prod}, for functions $u\in H^{\ov{\smT}}$ and $v\in H^{\ov{\sdmT}}$we have
\begin{equation*}
\dint_0^{T}(\Dt u,v)_{H}=-(u^0,v^{\frac{1}{2}})_{H}+(u^{M},v^{M+\frac{1}{2}})_{H}-\int_{0}^{T}(\Dtbar v,u)_{H}.
\end{equation*}
A summary with useful formulas and estimates are presented in \Cref{app:discrete_things}.

\section{A relaxed time-discrete observability inequality}\label{sec:relaxed_obs}

As classical in control theory, the controllability of \eqref{eq:ks_heat} can be reduced to study the observability of its adjoint system, which in this case is given by
\begin{equation}\label{eq:adj_ks_heat}
\begin{cases}
\D \frac{p^{n-\frac12}-p^{n+\frac12}}{\dt}-\Gamma \difx^2 p^{n-\frac12}-c\difx p^{n-\frac12}={ q^{n-\frac12}} &n\in\inter{1,M}, \\
\D \frac{q^{n-\frac12}-q^{n+\frac12}}{\dt}+\gamma \difx^4 q^{n-\frac12}-\difx^3 q^{n-\frac12}+a\difx^2 q^{n-\frac12}={ p^{n-\frac12}}&n\in\inter{1,M}, \\
\left(p_{|\partial\Omega}\right)^{n-\frac12} = \left(q_{|\partial\Omega}\right)^{n-\frac12}=\left(\difx q_{|\partial\Omega}\right)^{n-\frac{1}{2}}=0 &n\in\inter{1,M}, \\
p^{M+\frac12}=p_T, \quad q^{M+\frac12}=q_T.
\end{cases}
\end{equation}

Since we are interested in small-time controllability, we shall always consider that $T\in(0,1)$. The main goal of this section is to prove the following.

\begin{prop}\label{eq:prop_obs}
For  and any $\dt>0$ small enough, there exist positive constants $C_T$ and $C_1$ such that for every $(p_T,q_T)\in H_0^1(\Omega)\times H_0^2(\Omega)$, we have
\begin{equation}\label{eq:obs_final}
\snorme{p^{\frac12}}^2_{L^2(\Omega)}+\snorme{q^{\frac12}}^2_{L^2(\Omega)}\leq C_T^2\left(\ddbint_{\omega\times(0,T)}|q|^2+{e^{-\frac{C_1}{(\dt)^{1/10}}}}\left[\snorme{\difx p_T}^2_{L^2(\Omega)}+\snorme{\difx^2 q_T}^2_{L^2(\Omega)}\right]\right),
\end{equation}
where $C_T=e^{C/T}$ and the constants $C$, $C_1$ only depend on $\Omega$, $\omega$, $m$, $\Gamma$, $\gamma$, $a$, and $c$
\end{prop}

Estimate \eqref{eq:obs_final} is precisely a \textit{relaxed} observability inequality. It is weaker than a regular observability inequality due to the presence of the extra terms on the right-hand side. The study of relaxed observability estimates for discretized
parabolic equations was initiated in \cite{LT06}. We refer to \cite{Boy13} for a further review and discussion. To obtain \eqref{eq:obs_final}, we will use time-discrete Carleman estimates and well-known methodologies adapted to the discrete setting.

\subsection{Preliminaries on Carleman estimates}

In this part, we recall a known time-discrete Carleman estimate for the heat equation and present a new Carleman estimate for the fourth-order parabolic operator which, as mentioned before, is one of the main contributions of this paper.  

Let $\omega_0\Subset \omega$ and consider a function $\beta\in C^4([0,1])$ satisfying 
\begin{equation}\label{eq:cond_beta1}
\begin{gathered}
\beta(x)>0 \quad\forall x\in(0,1),  \\
\beta(0)=\beta(1)=0,
\end{gathered}
\end{equation}
and
\begin{equation}\label{eq:cond_beta2}
|\beta^\prime(x)|\geq \delta > 0 \quad \forall x\in[0,1]\setminus\omega_0 \ \text{for some } \delta >0.	
\end{equation}
The existence and construction of such function is classical, see, e.g. \cite[Lemma 2.1]{FI96}. We observe that \eqref{eq:cond_beta1} and \eqref{eq:cond_beta2} imply that
\begin{equation}\label{eq:positive_beta}
\beta^\prime(0)>0 \quad\text{and}\quad \beta^\prime(1)<0.
\end{equation}

In the spirit of \cite{Gue07}, for a parameter $\lambda>1$, we define the weight function
\begin{equation}\label{eq:def_varphi}
\varphi_m(x)=e^{\lambda(c_2+\beta(x))}-e^{\lambda c_1},
\end{equation}
where $c_2= k\norme{\beta}_\infty$ and $c_1=e^{\frac{k(m+1)}{m}\norme{\beta}_\infty}$. Here, $k>m>0$ are such that $\varphi_m<0$. We also introduce the weight
\begin{equation}\label{eq:def_theta}
\theta_m(t)=\frac{1}{(t+\delta T)^m(T+\delta T-t)^m},
\end{equation}
where $0<\delta<1/2$. The parameter $\delta$ is introduced to avoid singularities at time $t=0$ and $t=T$. These singularities, which correspond to the case $\delta=0$, are exploited in the continuous case, but are rather difficult to handle at the discrete level (for further discussion, see e.g. \cite[p. 6 and Remark 1.5]{BHS20}).  %Since the parameter $m$ is fixed from the beginning, hereinafter we shall write $\varphi$ and $\theta$ instead of $\varphi_m$ and $\theta_m$ to avoid cumbersome notation.

To avoid cumbersome notation and since at some point the parameter $m$ will be fixed, we simply write $\varphi$ and $\theta$ instead of $\varphi_m$ and $\theta_m$.

In the remainder of this section, we use the following notation to abridge the estimates
\begin{align*}
\mathcal I_H(p)&=\tau^{-1}\dbint_{Q}\tbm(e^{2\tau\theta\varphi}\theta^{-1})\left(|\Dtbar p|^2+|\partial_x^2\taubm{p}|^2\right)+\tau\dbint_{Q}\tbm(e^{2\tau\theta\varphi}\theta)|\difx\taubm{p}|^2\\
&\quad +\tau^3\dbint_{Q}\tbm(e^{2\tau\theta\varphi}\theta^3)\taubm{p}^2, \\
\mathcal I_{KS}(q)&=\tau^{-1}\dbint_{Q}\tbm(e^{2\tau\theta\varphi}\theta^{-1})\left(|\Dtbar q|^2+|\partial_x^4\taubm{q}|^2\right)+\tau\dbint_{Q}\tbm(e^{2\tau\theta\varphi}\theta)|\difx^3\taubm{q}|^2 \\ \notag
&\quad + \tau^3\dbint_{Q}\tbm(e^{2\tau\theta\varphi}\theta^3)|\difx^2\taubm{q}|^2+\tau^5\dbint_{Q} \tbm(e^{2\tau\theta\varphi}\theta^5)|\difx\taubm{q}|^2  +\tau^7\dbint_{Q} \tbm(e^{2\tau\theta\varphi}\theta^7)\taubm{q}^2,
\end{align*}
and
\begin{align*}
&\mathcal W_{H}(p) = \sum_{n\in\{0,M\}}\int_0^{1} \left|\left(e^{s\varphi} p\right)^{n+\frac12}\right|^2+ \int_{0}^{1} \left|\left(e^{s\varphi}\difx p\right)^{M+\frac12}\right|^2, \\
&\mathcal W_{KS}(q) = \sum_{n\in\{0,M\}}\left(\int_0^{1} \left|\left(e^{s\varphi} q\right)^{n+\frac12}\right|^2 + \int_0^{1} \left|\left(e^{s\varphi} \difx q\right)^{n+\frac12}\right|^2  \right)+ \int_{0}^{1} \left|\left(e^{s\varphi}\difx^2 q\right)^{M+\frac12}\right|^2.
\end{align*}

We state a uniform Carleman estimate for the time-discrete backward second-order parabolic operator defined on the dual grid 
\begin{equation}\label{eq:parab_op}
(L_{\sdmT}p)^n=-(\Dtbar p)^n-\Gamma \difx^2\taubm{p}^n, \quad n\in\inter{1,M}. 
\end{equation}
\begin{lem}\label{lem:carleman_heat}
 Let $m\geq 1$ and let $\varphi$ and $\theta$ be defined as in \eqref{eq:def_varphi} and \eqref{eq:def_theta} with $\beta$ verifying \eqref{eq:cond_beta1}--\eqref{eq:cond_beta2}. For the parameter $\lambda\geq 1$ sufficiently large, there exist $C>0$, $\tau_0\geq 1$, $\epsilon_0>0$, depending on $\omega$, $\lambda$ , $\Gamma$, and $m$, such that
\begin{align}\notag 
  \mathcal I_{H}(p)
&\leq C\left(\dbint_{Q}\tbm(e^{2\tau\theta\varphi})|L_{\sdmT}p|^2+\tau^3\dbint_{\omega\times(0,T)}\tbm(e^{2\tau\theta\varphi}\theta^3)\taubm{p}^2\right) +C(\dt)^{-1}\mathcal W_{H}(p)
\end{align}
for all $\tau\geq \tau_0(T^{2m}+T^{2m-1})$, and for all $\dt>0$ and $0<\delta\leq 1/2$ satisfying the condition
\begin{equation}\label{eq:cond_delta_calor}
\frac{\dt\tau^{4}}{\delta ^{4m} T^{6m}} \leq \epsilon_0, 
\end{equation}
and $p$ is any time-discrete function in $p\in (H^2(0,1)\cap H_0^1(0,1))^{\ov{\sdmT}}$.
\end{lem}

The proof \Cref{lem:carleman_heat} is essentially given in \cite{BHS20}. Actually, the authors prove this result by considering a slightly different weight, more precisely, they take $c_2=0$ and $c_1=K>\|\beta\|_{\infty}$ with $K$ large enough. Nonetheless, a close inspection shows that the proof can be adapted to our case just by considering that 

\begin{equation}
\theta^\prime(t)\leq CT \theta^{1+1/m}  \quad\text{and}\quad \max_{t\in[0,T]}\theta(t)\leq (\delta^mT^{2m})^{-1}.
\end{equation} 

Even though it does not explicitly appear in the Carleman estimate since it is hidden in the definition of \eqref{eq:def_theta}, the parameter $\delta$ plays a crucial role. On one hand, it is possible to define continuously the weight outside $[0,T]$, since functions on the dual mesh $\ov{\dmT}$ have one extra point lying outside this interval. On the other, it allows us to estimate the discrete derivative of $\theta$ (see \Cref{app:theta_comps}) and perform several related computations.

Now, we present a new uniform Carleman estimate for the time-discrete backward fourth-order parabolic operator defined on the dual grid, that is,
\begin{equation}\label{eq:4th_order_op}
(P_{\sdmT}q)^n=-(\Dtbar{q})^n+\gamma \partial_x^4\taubm{q}^n, \quad n\in\inter{0,M},
\end{equation}
where $q\in (H^4(0,1))^{\ov{\sdmT}}$. We have the following.
\begin{lem}\label{thm:time_discrete_fourth}
Let $m\geq 1/3$ and let $\varphi$ and $\theta$ be defined as in \eqref{eq:def_varphi} and \eqref{eq:def_theta} with $\beta$ verifying \eqref{eq:cond_beta1}--\eqref{eq:cond_beta2}. For the parameter $\lambda\geq 1$ sufficiently large, there exist $C>0$, $\tau_0\geq 1$, $\epsilon_0>0$, depending on $\omega$, $\lambda$, $\gamma$, and $m$, such that
\begin{align}\notag 
  \mathcal I_{KS}(q)
&\leq C\left(\dbint_{Q}\tbm(e^{2\tau\theta\varphi})|P_{\sdmT}q|^2+\tau^7\dbint_{\omega\times(0,T)}\tbm(e^{2\tau\theta\varphi}\theta^7)\taubm{q}^2\right) +C(\dt)^{-1}\mathcal W_{KS}(q)
\end{align}
for all $\tau\geq \tau_0(T^{2m}+T^{2m-1/3})$, and for all $\dt>0$ and $0<\delta\leq 1/2$ satisfying the condition
{
\begin{equation}\label{eq:cond_delta}
\frac{\dt\tau^{5}}{\delta ^{10m} T^{14m}} \leq \epsilon_0. 
\end{equation}
}
and $q$ is any time-discrete function in $q\in (H^4(0,1)\cap H_0^2(0,1))^{\ov{\sdmT}}$.
\end{lem}

The proof of this result follows the steps presented in \cite[Section 2]{BHS20} for the time-discrete Carleman estimate for second-order parabolic operators, together with some new computations specific to the fourth-order case. We give an abridged proof in \Cref{app:car_fourth}

We readily observe that this new estimate has the same structure as in the continuous case (cf. \cite[Theorem 3.3]{CMP15}) but due to the discrete nature of the problem it has some remainder terms like the Carleman estimate for the second-order parabolic operator. Note that in this case, we have several extra terms corresponding to first and second order derivatives of the function $q$. As explained in \cite{BHS20}, it is not clear if these terms can be removed, even for the heat equation.  Lastly, as discussed in \cite[Remark 3.4]{CC16}, $m=1/3$ is the actual optimal power and improves the estimation of the cost of control  of the linear KS equation (see \cite{CG16}). The same remark applies here.

\subsection{A carleman estimate with only one observation}

We present a Carleman estimate for the coupled adjoint system \eqref{eq:adj_ks_heat} with only one observation on the right-hand side. Using the notation of \Cref{sec:not_time}, we compactly rewrite \eqref{eq:adj_ks_heat} as

\begin{equation}\label{eq:adj_ks_heat_compact}
\begin{cases}
\D -(\Dtbar p)-\Gamma \difx^2 \taubm{p}^{n}-c\difx \taubm{p}^{n}=\taubm{q}^{n} &n\in\inter{1,M}, \\
\D -(\Dtbar q)+\gamma \difx^4 \taubm{q}^{n}-\difx^3 \taubm{q}^{n}+a\difx^2 \taubm{q}^{n}=\taubm{p}^{n}&n\in\inter{1,M}, \\
%\left(p_{|\partial\Omega}\right)^{n-\frac12} = \left(q_{|\partial\Omega}\right)^{n-\frac12}=\left(\difx q_{|\partial\Omega}\right)^{n-\frac{1}{2}}=0 &n\in\inter{1,M}, \\
p^{M+\frac12}=p_T, \quad q^{M+\frac12}=q_T.
\end{cases}
\end{equation}

We have the following.
\begin{prop}\label{prop:car_one_obs} Let $m\geq 1$ and let $\varphi$ and $\theta$ be defined as in \eqref{eq:def_varphi} and \eqref{eq:def_theta} with $\beta$ verifying \eqref{eq:cond_beta1}--\eqref{eq:cond_beta2}. For the parameter $\lambda\geq 1$ sufficiently large, there exist $C>0$, $\tau_1\geq 1$, $\epsilon_1>0$, depending on $\omega$, $\lambda$, $\Gamma$, $\gamma$, and $m$, such that for any $(p_T,q_T)\in H_0^1(\Omega)\times H_0^2(\Omega)$, the solution $(p,q)$ to \eqref{eq:adj_ks_heat_compact} satisfies
\begin{align}\label{eq:car_with_one_obs}
\mathcal I_{H}&(p)+\mathcal I_{KS}(q)  \leq C\left(\dbint_{\omega_1\times(0,T)}  \tbm(e^{2s\varphi}s^{39})\taubm{q}^2  \right) + C(\dt)^{-1}\Big(\mathcal W_{H}(p)+ \mathcal W_{KS}(q)\Big),
\end{align}
for all $\tau\geq \tau_1(T^{2m}+T^{2m-1}+T^{2m-1/3})$, and for all $\dt>0$ and $0<\delta\leq 1/2$ verifying the condition
\begin{equation}\label{eq:cond_delta_car_single}
\frac{\dt\tau^{10}}{\delta ^{10m} T^{20m}} \leq \epsilon_1.
\end{equation}
\end{prop}

\begin{proof}
We follow the classical methodology of \cite{GBdT10} (see also \cite{CMP15}) adapted to the time-discrete case. For clarity, we have divided the proof in three steps. 

\smallskip
\textit{Step 1. First estimates.} We apply the corresponding Carleman estimate to each equation in the system. Note that the condition for $\epsilon_0$ is stronger for the fourth-order equation, thus conditions \eqref{eq:cond_delta_calor} and \eqref{eq:cond_delta} can be put in a single one.  More precisely, adding up the estimates, we obtain
\begin{align*}\notag
&I_{H}(p)+I_{KS}(q) \\
&\leq C\left(\dbint_{\omega_0\times(0,T)}\left[\tbm(e^{2s\varphi }s^3 p^2)+\tbm(e^{2s\varphi}s^7q^2)\right] + \dbint_{Q} \taubm{e^{2s\varphi}} \left|c\difx\taubm{p}^2+\taubm{q}\right|^2 \right) \\
&\quad +C\left(\dbint_{Q}\taubm{e^{2s\varphi}}\left|-a\difx^2\taubm{q}+\difx^3\taubm{q}-\taubm{p}\right|^2 \right) + C(\dt)^{-1}\Big(\mathcal W_{H}(p)+ \mathcal W_{KS}(q)\Big),
\end{align*}
for all $\tau \geq \tau_0(T^{2m}+T^{2m-1}+T^{2m-1/3})$ and $\dt \tau^5(\delta^{10m}T^{14m})^{-1}\leq \epsilon_0$. Since $1\leq CT^{2m}\theta$, we can choose $\tau_1\geq \tau_0$ large enough (depending on $a$ and $c$) such that
\begin{align}\notag
&I_{H}(p)+I_{KS}(q) \\ \label{eq:car_coup_init}
&\quad \leq C\left(\dbint_{\omega_0\times(0,T)}\left[\tbm(e^{2s\varphi }s^3 p^2)+\tbm(e^{2s\varphi}s^7q^2)\right]\right) + C(\dt)^{-1}\Big(\mathcal W_{H}(p)+ \mathcal W_{KS}(q)\Big),
\end{align}
for all $\tau \geq \tau_1(T^{2m}+T^{2m-1}+T^{2m-1/3})$ and 
\begin{equation}\label{eq:small_cond_init}
\frac{\dt \tau^5}{\delta^{10m}T^{14m}} \leq \epsilon_0.
\end{equation}

\smallskip
\textit{Step 2. A local energy estimate.} In this step, we will obtain an estimate for the local term for $p$ in terms of several observations for $q$ and its spatial derivatives. Let $\omega_0\Subset \omega_1\Subset \omega$ and let $\eta\in C_c^\infty(\omega_1)$ such that $\eta=1$ in $\omega_0$. From the equation verified by $p$ we have
\begin{align}\notag 
\dbint_{\omega_0\times(0,T)}  \tbm(e^{2s\varphi}s^3)\taubm{p}^2 &\leq \dbint_{\omega_1\times(0,T)} \eta \tbm(e^{2s\varphi}s^3)\taubm{p}\left(-\Dtbar q + \difx^4\taubm{q}+\difx^3 \taubm{q}+a\difx^2 \taubm{q}\right)  \\ \label{eq:est_local_inicial}
& =: I_1+I_2
\end{align}
with 
\begin{align*}
I_1&=\dbint_{\omega_1\times(0,T)} \eta\, \tbm(e^{2s\varphi}s^3) \taubm{p}\left[\difx^4\taubm{q}+\difx^3\taubm{q}+a\difx^2\taubm{q}\right], \\
I_2&=-\dbint_{\omega_1\times(0,T)}\eta\, \tbm(e^{2s\varphi}s^3)\taubm{p}\Dtbar q.
\end{align*}

We note that $I_1$ does not involve time-discrete derivatives. Thus, using the estimate $\left|\difx (e^{2s\varphi} s^3\eta)\right|\leq C e^{2s\varphi}s^4$ in $\omega_1\times(0,T)$, we can integrate by parts in the fourth-order term and use Cauchy-Schwarz and Young inequalities to get
\begin{align}\notag
|I_1| &\leq \varrho \left(\dbint_{Q} \left[\tbm(e^{2s\varphi}s)|\difx\taubm{p}|^2 + \tbm(e^{2s\varphi}s^3)\taubm{p}^2 \right] \right) \\ \label{eq:est_I1}
&\quad + \frac{C}{\varrho} \left(\dbint_{\omega_1\times(0,T)} \left[\tbm(e^{2s\varphi}s^3)|\difx^2\taubm{q}|^2+\tbm(e^{2s\varphi}s^5)|\difx^3 \taubm{q}|^2 \right] \right),
\end{align}
for any $0<\varrho<1$ and some $C>0$ only depending on $\omega_1$ and a.

To bound $I_2$ we argue as follows. Discrete integration by parts (see formula \eqref{by_parts_same}) yields
\begin{align}\notag
I_2&=\left((e^{s\varphi}s^{3/2}\eta^{1/2}p)^{\frac{1}{2}}, (e^{s\varphi}s^{3/2}\eta^{1/2}q)^{\frac{1}{2}}\right)_{L^2(\Omega)}-\left((e^{s\varphi}s^{3/2}\eta^{1/2}p)^{M+\frac{1}{2}}, (e^{s\varphi}s^{3/2}\eta^{1/2}q)^{M+\frac{1}{2}}\right)_{L^2(\Omega)} \\ \label{eq:def_I2}
&\quad + \dbint_{Q}\taubp{q}\Dtbar(e^{2s\varphi}s^3 p) \eta := \sum_{j=1}^{3} I_2^{(j)}.
\end{align}

A direct computation gives
\begin{equation}\label{eq:est_I21_I22}
|I_2^{(1)}|+|I_2^{(2)}| \leq C\sum_{n\in\{0,M\}}\left(\left|(e^{s\varphi}s^{3/2}p)^{n+\frac12}\right|^2_{L^2(\Omega)}+\left|(e^{s\varphi}s^{3/2}q)^{n+\frac12}\right|^2_{L^2(\Omega)}\right).
\end{equation}
On the other hand, using formula \eqref{deriv_prod}, we have
\begin{equation*}
I_2^{(3)}=\dbint_{Q} \eta\, \tbm(e^{2s\varphi}s^3)\taubp{q} \Dtbar p  + \dbint_{Q} \eta \Dtbar(e^{2s\varphi}s^3)\taubp{p}\taubp{q},
\end{equation*}
whence, using that $\taubp{q}=\taubm{q}+\dt \Dtbar q$ and $\Dtbar(e^{2s\varphi}\theta^3)=\taubp{e^{2s\varphi}}\Dtbar(\theta^3)+\taubm{\theta}^3\Dtbar(e^{2s\varphi})$ (which follows from \eqref{deriv_prod}), we get
\begin{align}\notag
I_2^{(3)} &= \dbint_{Q}\eta\, \tbm(e^{2s\varphi}s^3) \taubm{q}\Dtbar p + \dt \dbint_{Q} \eta\, \tbm(e^{2s\varphi}s^3)(\Dtbar p)(\Dtbar q) \\ \label{eq:def_I23}
& + \dbint_{Q} \eta\, \taubp{e^{2s\varphi }}\taubp{p}\taubp{q} \tau^3 \Dtbar(\theta^3) + \dbint_{Q} \eta\, \taubp{p}\taubp{q} \taubm{s}^3 \Dtbar(e^{2s\varphi}):= \sum_{i=1}^{4} H_i.
\end{align}

Using Cauchy-Schwarz and Young inequalities, we readily get
\begin{align*}
|H_1|+|H_2| &\leq \varrho \dbint_{Q} \tbm(e^{2s\varphi}s^{-1}) (\Dtbar p)^2 + \frac{\dt}{2}\dbint_{Q} \tbm(e^{2s\varphi}s^3)\left[(\Dtbar p)^2+(\Dtbar q)^2\right] \\ 
&\quad + \frac{C}{\varrho}\dbint_{\omega_1\times(0,T)} \tbm(e^{2s\varphi}s^7 )\taubm{q}^2.
\end{align*}
Now, let us choose $\epsilon_1>0$ small enough such that $\epsilon_1\leq \min\{\epsilon_0, 1/2C\}$ where $C>0$ is the constant appearing in \eqref{eq:car_coup_init}. Hence, setting 
\begin{equation}\label{eq:cond_epsilon_1}
\frac{\dt \tau^4}{\delta^{4m}T^{8m}}\leq \epsilon_1,
\end{equation}
we have
\begin{align}\notag 
|H_1|+|H_2| &\leq \varrho \dbint_{Q} \tbm(e^{2s\varphi}s^{-1}) (\Dtbar p)^2 + \epsilon_1 \dbint_{Q} \tbm(e^{2s\varphi}s^{-1})\left[(\Dtbar p)^2+(\Dtbar q)^2\right] \\ \label{H1_H2}
&\quad + \frac{C}{\varrho}\dbint_{\omega_1\times(0,T)} \tbm(e^{2s\varphi}s^7 )\taubm{q}^2.
\end{align}

For the third term, using formula \ref{est_dt_square} with $\ell=3$ and the facts that $\theta^{-1}(t)\leq CT^{2m}$ and $\tau\geq CT^{2m-1}$ we have
\begin{equation*}
|H_3|\leq \dbint_{Q} \eta |\taubp{p}||\taubm{q}| \tbp(e^{2s\varphi}s^4) + \dbint_{Q} \eta |\taubp{p}||\taubm{q}| \taubp{e^{2s\varphi}}\frac{\dt \tau^3}{\delta^{3m+2}T^{6m+2}}
\end{equation*}
Since $m\geq1/3$, provided
\begin{equation}\label{eq:dt_small_1}
\frac{\dt \tau^3}{\delta^{9m}T^{12m}}\leq 1
\end{equation}
we can use the properties of the function $\eta$ to deduce
\begin{equation*}
|H_3|\leq 2\varrho \dbint_{Q}\tbp(e^{2s\varphi} s^3)\taubp{p}^2+ \frac{C}{\varrho}\dbint_{\omega_1\times(0,T)}\tbp(e^{2s\varphi}s^5)\taubp{q}^2
\end{equation*}
for any $0<\varrho<1$ and where we have used that $s(t)^{-1}\leq 1$ for $t\in[0,T]$ by increasing (if necessary) the value of $\tau_1$. By shifting the above integrals (see formula \eqref{trans_doub2}), we get
\begin{align}\notag
|H_3| &\leq 2\varrho \dbint_{Q}\tbp(e^{2s\varphi} s^3)\taubp{p}^2+ \frac{C}{\varrho}\dbint_{\omega_1\times(0,T)}\tbp(e^{2s\varphi}s^5)\taubp{q}^2  \\ \label{H3}
&\quad + C \dt \left(\left|(e^{s\varphi}s^{3/2}p)^{M+\frac12}\right|^2_{L^2(\Omega)}+\left|(e^{s\varphi}s^{5/2}q)^{M+\frac12}\right|^2_{L^2(\Omega)}\right).
\end{align}

From \Cref{lem:deriv_lemma_time} and \Cref{rmk:shifts_equiv}, we can compute 
\[
|\Dtbar(e^{2s\varphi})|\leq C \left(\taup{e^{2s\varphi}}\taubp{s}^2+\taubp{e^{2s\varphi}}\frac{\dt \tau^2}{\delta^{2m+2}T^{4m+2}}\right).
\]
This, together with Lemma \ref{shift} and arguments similar to the ones used for estimating $H_3$, yield that
\begin{align}\notag
|H_4|&\leq 4 \varrho\dbint_{Q}\tbm(e^{2s\varphi}s^3)\taubm{p}^2+ \frac{C}{\varrho} \dbint_{\omega_1\times(0,T)} \tbm(e^{2s\varphi}s^7)\taubm{p}^2 \\ \label{H4}
&\quad + C \dt \left(\left|(e^{s\varphi}s^{3/2}p)^{M+\frac12}\right|^2_{L^2(\Omega)}+\left|(e^{s\varphi}s^{7/2}q)^{M+\frac12}\right|^2_{L^2(\Omega)}\right)
\end{align}
provided \eqref{eq:dt_small_1} holds. Having reached this point, we observe that the smallness conditions \eqref{eq:small_cond_init}, \eqref{eq:cond_epsilon_1}, and \eqref{eq:dt_small_1} can be combined into a single one verifying
\begin{equation}\label{eq:small_final}
\frac{\dt \tau^{10}}{\delta^{10m}T^{20m}} \leq \epsilon_1
\end{equation}
with $\epsilon_1>0$ small enough as above. We have chosen these particular powers for $\tau$, $\delta$ and $T$ since they simplify the computations in the following section.

To conclude this step, we just have to combine expressions \eqref{eq:est_I1}--\eqref{eq:def_I23}, \eqref{H1_H2}, \eqref{H3}--\eqref{H4} into \eqref{eq:est_local_inicial} to obtain
\begin{align}\notag 
\dbint_{\omega_0\times(0,T)}&  \tbm(e^{2s\varphi}s^3)\taubm{p}^2 \\ \notag
&\leq 8 \varrho\left(\dbint_{Q} \left[\tbm(e^{2s\varphi}s)|\difx\taubm{p}|^2 + \tbm(e^{2s\varphi}s^3)\taubm{p}^2 + \tbm(e^{2s\varphi}s^{-1})(\Dtbar p)^2 \right] \right) \\ \notag
&\quad +  \epsilon_1 \dbint_{Q} \tbm(e^{2s\varphi}s^{-1})\left[(\Dtbar p)^2+(\Dtbar q)^2\right] \\ \notag
&\quad + \frac{C}{\varrho} \left(\dbint_{\omega_1\times(0,T)}\tbm(e^{2s\varphi}s^3)|\difx^2\taubm{q}|^2+\tbm(e^{2s\varphi}s^5)|\difx^3 \taubm{q}|^2+\tbm(e^{2s\varphi}s^7) \taubm{q}^2 \right) \\ \label{eq:est_local_final}
&\quad + C (\dt)^{-1} \left(\left|(e^{s\varphi}p)^{M+\frac12}\right|^2_{L^2(\Omega)}+\left|(e^{s\varphi}q)^{M+\frac12}\right|^2_{L^2(\Omega)}\right)
\end{align}
for any $0<\varrho<1$. In the last line, we have used \eqref{eq:small_final} to remove the powers of $s$.

\smallskip
\textit{Step 3. Last arrangements and conclusion.}
Using estimate \eqref{eq:est_local_final} in inequality \eqref{eq:car_coup_init}, taking $\varrho>0$ small enough and recalling the definition of $\epsilon_1$, we readily get
\begin{align}\notag
I_{H}&(p)+I_{KS}(q) \\ \notag
& \leq C\left(\dbint_{\omega_1\times(0,T)} \left[ \tbm(e^{2s\varphi}s^3)|\difx^2\taubm{q}|^2+\tbm(e^{2s\varphi}s^5)|\difx^3 \taubm{q}|^2+\tbm(e^{2s\varphi}s^7) \taubm{q}^2 \right] \right) \\ \label{eq:est_sin_local_p}
&\quad + C(\dt)^{-1}\Big(\mathcal W_{H}(p)+ \mathcal W_{KS}(q)\Big),
\end{align}
for all $\tau \geq \tau_1(T^{2m}+T^{2m-1}+T^{2m-1/3})$ and provided \eqref{eq:small_final} holds. Notice that we have removed the local term of $p$ at the price of having observations of the space derivatives of $q$.

Now, we claim that for any $\vartheta>0$, there exists a constant $C_{\vartheta}>0$ uniform with respect to $\dt$ such that 
\begin{align}\notag 
&\dbint_{\omega_1\times(0,T)} \left[ \tbm(e^{2s\varphi}s^3)|\difx^2\taubm{q}|^2+\tbm(e^{2s\varphi}s^5)|\difx^3 \taubm{q}|^2+\tbm(e^{2s\varphi}s^7) \taubm{q}^2 \right] \\ \notag
&\leq \vartheta \left( \dbint_{Q} \left[ \tbm(e^{2s\varphi}s^{-1})|\difx^4\taubm{q}|^2+\tbm(e^{2s\varphi}s)|\difx^3 \taubm{q}|^2+\tbm(e^{2s\varphi}s^3) |\difx^2\taubm{q}|^2 \right] \right) \\ \label{eq:est_local_phi}
&\quad + C_{\vartheta} \dbint_{\omega\times(0,T)}{\tbm(e^{2s\varphi}s^{39})}\taubm{q}^2.
\end{align} 
If this inequality holds, we just have to use \eqref{eq:est_local_phi} in \eqref{eq:est_sin_local_p} and take $\vartheta>0$ sufficiently small to discover the Carleman estimate \eqref{eq:car_with_one_obs}.

To prove \eqref{eq:est_local_phi} we follow a standard procedure. Let $\omega_2$ an open set such that $\omega_1\Subset \omega_2\Subset \omega$ and $\zeta\in C_c^\infty(\omega_2)$ with $\zeta=1$ in $\omega_1$. Then,
\begin{align} \notag
\dbint_{\omega_1\times(0,T)} \tbm(e^{2s\varphi}s^5)|\difx^3\taubm{q}|^2 & \leq \dbint_{\omega_2\times(0,T)} \zeta \, \tbm(e^{2s\varphi}s^5) |\difx^3\taubm{q}|^2 \\ \notag
&=-\dbint_{\omega_2\times(0,T)} \zeta \, \tbm(e^{2s\varphi}s^5) \, \difx^4\taubm{q} \, \difx^2\taubm{q} \\ \label{eq:est_local_1}
&\quad + \frac{1}{2}\dbint_{\omega_2\times(0,T)} \tbm\left(\difx^2[\zeta e^{2s\varphi}s^5]\right) |\difx^2\taubm{q}|^2
\end{align}
and since $s$ only depends on time, in view of the estimate $|\difx^2(e^{2s\varphi})|\leq C e^{2s\varphi} s^2$ in $\omega_2\times(0,T)$, we have
\begin{align}\label{eq:est_local_2}
\dbint_{\omega_1\times(0,T)}&\tbm(e^{2s\varphi}s^5) |\difx^3\taubm{q}|^2 \leq \vartheta \dbint_{Q} \tbm(e^{2s\varphi} s^{-1})|\difx^4\taubm{q}|^2 + C_{\vartheta} \dbint_{\omega_2\times(0,T)} \tbm(e^{2s\varphi}s^{11})|\difx^2\taubm{q}|^2.
\end{align}
Analogously, if $\omega_3$ is an open set such that $\omega_2\Subset\omega_3\Subset \omega$, we get
\begin{align}\label{eq:est_local_3}
\dbint_{\omega_2\times(0,T)}&\tbm(e^{2s\varphi}s^{11}) |\difx^2\taubm{q}|^2 \leq \vartheta \dbint_{Q} \tbm(e^{2s\varphi} s)|\difx^3\taubm{q}|^2 + C_{\vartheta} \dbint_{\omega_3\times(0,T)} \tbm(e^{2s\varphi}s^{21})|\difx\taubm{q}|^2,
\end{align}
and finally
\begin{align}\label{eq:est_local_4}
\dbint_{\omega_3\times(0,T)}&\tbm(e^{2s\varphi}s^{21}) |\difx\taubm{q}|^2 \leq \vartheta \dbint_{Q} \tbm(e^{2s\varphi} s^3)|\difx^2\taubm{q}|^2 + C_{\vartheta} \dbint_{\omega\times(0,T)} \tbm(e^{2s\varphi}s^{39})\taubm{q}^2.
\end{align}
Thus, the claim follows from expressions \eqref{eq:est_local_1}--\eqref{eq:est_local_4}. This ends the proof. 
\end{proof}

\subsection{Proof of \Cref{eq:prop_obs}}

Now, we are in position to obtain the observability inequality \eqref{eq:obs_final}. The proof is divided into two steps: the first one resembles the continuous case and looks for estimating the weigt functions, while the second one is exclusive to the discrete case and connects the Carleman parameters with the discretization ones. 

In what follows, $C$ denotes a positive constant depending at most on $\Omega$, $\omega$, $m$, $\Gamma$, $\gamma$, $a$, and $c$, that may change from line to line. 

\smallskip
\textit{Step 1. Cleaning up the Carleman estimate.} In view of \eqref{trans_doub} and from our Carleman inequality with only one observation \eqref{eq:car_with_one_obs}, we have
\begin{equation}\label{eq:est_obs_init}
\ddbint_{Q}e^{2s\varphi} s^3 |p|^2 + \ddbint_{Q}e^{2s\varphi} s^7 |q|^2 \leq C\left(\ddbint_{\omega\times(0,T)}e^{2s\varphi}s^{39} |q|^2\right)+ C(\dt)^{-1}\Big(\mathcal W_{H}(p)+ \mathcal W_{KS}(q)\Big)
\end{equation}
for all $\tau\geq \tau_1(T^{2m}+T^{2m-1}+T^{2m-1/3})$ and any $\dt>0$ verifying \eqref{eq:cond_delta_car_single}. 

Now, we will remove the weights in inequality \eqref{eq:est_obs_init}. Since $\tau\geq 1$ and noting that $(e^{2s\varphi})^{n+\frac12}\geq e^{-\frac{2^{4m+1}\tau K_0}{3^m T^{2m}}}$ for $n\in\inter{M/4,3M/4}$, where $K_0:=\max_{x\in\ov{\Omega}}\{-\varphi(x)\}$, we see that the left hand-side of the above expression can be bounded as
\begin{equation}\label{eq:est_below_car}
\ddbint_{Q}e^{2s\varphi}s^3|p|^2+\ddbint_{Q}e^{2s\varphi}s^7|q|^2\geq  \sum_{n\in\inter{M/4,3M/4}}\dt \tau^3 e^{-\frac{C\tau}{T^{2m}}} T^{-6m} \left(|p^{n+\frac12}|^2_{L^2(\Omega)}+|q^{n+\frac12}|^2_{L^2(\Omega)}\right)
\end{equation}
for some $C>0$ uniform with respect to $\dt$. Using estimate \eqref{eq:est_normal} in \Cref{lem:est_method}, we see that after iteration
\begin{equation}\label{eq:disipation_normal}
\snorme{p^{\frac12}}^2_{L^2(\Omega)}+\snorme{q^{\frac12}}^2_{L^2(\Omega)}\leq e^{2CT} \left(\snorme{p^{n+\frac12}}^{2}_{L^2(\Omega)}+\snorme{q^{n+\frac12}}_{L^2(\Omega)}^2\right)
\end{equation}
for all $n\in\inter{1,M}$ and any $\dt>0$ such that $2C\dt<1$. Using this in the right-hand side of \eqref{eq:est_below_car} and summing over $n$, we have
\begin{align}\notag
\ddbint_{Q}e^{2s\varphi}s^3|p|^2+\ddbint_{Q}e^{2s\varphi}s^7|q|^2 &\geq \left(\frac{T}{2}-\dt\right) \tau^3 e^{-\frac{C\tau}{T^{2m}}-CT} T^{-6m} \left(\snorme{p^{\frac12}}^2_{L^2(\Omega)}+\snorme{q^{\frac12}}^2_{L^2(\Omega)}\right) \\ \label{eq:est_below_car_fin}
&\geq C T e^{-\frac{C\tau}{T^{2m}}-CT}  \left(\snorme{p^{\frac12}}^2_{L^2(\Omega)}+\snorme{q^{\frac12}}^2_{L^2(\Omega)}\right),
\end{align}
where we have used that $\tau\geq \tau_1 T^{2m}$.

Lets comeback to the right-hand side of \eqref{eq:est_obs_init}. Using \eqref{eq:cond_theta_plus}, we see that
\begin{align}\notag 
\mathcal W_{H}(p)+\mathcal W_{KS}(q) &\leq e^{-\frac{2^{m+1} k_0}{\delta^mT^{2m}}}\sum_{n\in\{0,M\}}\left(\snorme{p^{n+\frac12}}^2_{L^2(\Omega)}+\snorme{q^{n+\frac12}}^2_{L^2(\Omega)}+\snorme{\difx q^{n+\frac12}}^2_{L^2(\Omega)}\right) \\ \label{eq:est_boundaries_T}
&\quad + e^{-\frac{2^{m+1} k_0}{\delta^mT^{2m}}} \left(\snorme{\difx p^{M+\frac12}}^2_{L^2(\Omega)}+\snorme{\difx^2 q^{M+\frac12}}^2_{L^2(\Omega)}\right),
\end{align}
where $k_0:=\min_{x\in\ov\Omega}\{-\varphi(x)\}$. Arguing as we did for \eqref{eq:disipation_normal}, we can obtain from \eqref{eq:est_regular} and Poincar\'e inequality that 
\begin{equation}\label{eq:disipation_regular}
\snorme{\difx^2 q^{\frac12}}_{L^2(\Omega)}^2 \leq C e^{CT}\left(\snorme{\difx p^{n+\frac12}}^2_{L^2(\Omega)}+\|\difx^2 q^{n+\frac12}\|^2_{L^2(\Omega)}\right), \quad \forall n\in\inter{0,M-1},
\end{equation}
since $(p_{|\partial\Omega})^{n-\frac12}=(q_{|\partial\Omega})^{n-\frac12}=(\difx q_{|\partial\Omega})^{n-\frac12}=0$ for all $n\in\inter{1,M}$. By assumption $(p_T,q_T)\in H_0^1(\Omega)\times H_0^2(\Omega)$, thus \eqref{eq:disipation_normal}, \eqref{eq:est_boundaries_T} and \eqref{eq:disipation_regular} yield
\begin{equation}\label{eq:est_Ws}
\mathcal W_{H}(p)+\mathcal W_{KS}(q) \leq e^{-C\frac{\tau}{\delta^m T^{2m}}+CT} \left(\snorme{\difx p_T}_{L^2(\Omega)}^2+\snorme{\difx^2 q_{T}}^2_{L^2(\Omega)}\right)
\end{equation}
for some $C>0$ uniform with respect to $\dt$ and $\delta$. 

We observe that $e^{2s\varphi} s^{39}\leq \tau^{39}2^{78m}T^{-78m}\exp\left(-\frac{2^{2m+1}k_0}{T^{2m}}\right)\leq C$ for all $(x,t)\in Q$, uniformly for $\tau\geq \frac{39}{2^{m+1}k_0}T^{2m}$. Hence, putting together \eqref{eq:est_obs_init}, \eqref{eq:est_below_car_fin} and \eqref{eq:est_Ws}, we get
\begin{align*}
&\snorme{p^{\frac12}}_{L^2(\Omega)}^2+\snorme{q^{\frac12}}_{L^2(\Omega)}^2 \\
&\quad \leq CT^{-1} e^{CT}\left( e^{\frac{C\tau}{T^{2m}}} \ddbint_{\omega\times(0,T)}|q|^2+(\dt)^{-1} e^{\frac{\tau}{T^{2m}}(C-\frac{C}{\delta^m})}\left[\snorme{\difx p_T}_{L^2(\Omega)}^2+\snorme{\difx^2 q_{T}}^2_{L^2(\Omega)}\right]  \right)
\end{align*}
for all $\tau\geq \tau_2(T^{2m}+T^{2m-1}+T^{2m-1/3})$ where $\tau_2\geq \max\{\tau_1,39/(2^{m+1}k_0)\}$. For $0<\delta \leq \delta_1<1/2$ with $\delta_1$ small enough, we get
\begin{equation}\label{eq:obs_casi}
\snorme{p^{\frac12}}_{L^2(\Omega)}^2+\snorme{q^{\frac12}}_{L^2(\Omega)}^2 \leq e^{C/T}\left( e^{\frac{C\tau}{T^{2m}}} \ddbint_{\omega\times(0,T)}|q|^2+(\dt)^{-1}e^{-\frac{C\tau}{\delta^{m}T^{2m}}} \left[\snorme{\difx p_T}_{L^2(\Omega)}^2+\snorme{\difx^2 q_{T}}^2_{L^2(\Omega)}\right]  \right).
\end{equation}
\smallskip

\textit{Step 2. Connection of the discrete parameters.} To conclude the proof, we will connect the parameters associated to the discretization, i.e., $\dt$ and $\delta$. We recall that the condition 
$
\frac{\dt \tau^{10}}{\delta^{10m}T^{20m}}\leq \epsilon_1
$
should be fulfilled along with $0<\delta\leq \delta_1$ and $2C\dt<1$ for some $C>0$. 

Let us fix $\tau=\tau_2(T^{2m}+T^{2m-1}+T^{2m-1/3})$ and define $\widetilde{\dt}:=\epsilon_0\frac{\delta_1^{10m}}{\tau_2^{10}}(1+\frac{1}{T}+\frac{1}{T^{1/3}})^{-10}$. Hence, $\frac{\tau^{10} \widetilde{\dt}}{\delta_1^{10m}T^{20m}}=\epsilon_0$. Now, we choose $\dt \leq \min\{\widetilde{\dt},1/2C\}$ and set 
$
\delta=\frac{(\dt)^{1/10m}\delta_1}{(\widetilde{\dt})^{1/10m}}\leq \delta_1.
$
We find that $\frac{\dt \tau^{10}}{\delta^{10m}T^{20m}}=\epsilon_0$ and $\tau/(\delta^mT^{2m})= \epsilon_0^{1/10}/(\dt)^{1/10}$, hence from \eqref{eq:obs_casi} we get
\begin{align*}
\snorme{p^{\frac12}}_{L^2(\Omega)}^2&+\snorme{q^{\frac12}}_{L^2(\Omega)}^2 \\
&\leq e^{C/T}\left( e^{C(1+{1}/{T}+{1}/{T^{1/3}})} \ddbint_{\omega\times(0,T)}|q|^2+(\dt)^{-1}e^{-\frac{C \epsilon_0^{1/10}}{(\dt)^{1/10}}} \left[\snorme{\difx p_T}_{L^2(\Omega)}^2+\snorme{\difx^2 q_{T}}^2_{L^2(\Omega)}\right]  \right), \\
&\leq e^{C/T}\left(\ddbint_{\omega\times(0,T)}|q|^2+e^{-\frac{C_1}{(\dt)^{1/10}}} \left[\snorme{\difx p_T}_{L^2(\Omega)}^2+\snorme{\difx^2 q_{T}}^2_{L^2(\Omega)}\right]  \right)
\end{align*}
for some $C_1>0$ uniform with respect to $\dt$  and where we have used that $T\in(0,1)$ to simplify in the first term. This ends the proof.

\section{$\phi(\dt)$-controllability}\label{sec:phi_dt_contr}

This section is devoted to prove \Cref{thm:main_1}. Due to the presence of the spatial derivatives of the initial data in the right-hand of \eqref{eq:obs_final}, we have to prove \Cref{thm:main_1} in two steps. First, using the well-known penalized Hilbert Uniqueness Method (see, for instance, \cite{GLH08,Boy13}) we build a time-discrete control providing a controllability result in the space $H^{-1}(\Omega)\times H^{-2}(\Omega)$. Then, using classical elliptic arguments, we obtain the controllability in the $L^2$-setting.

\subsection{Controllability in $H^{-1}\times H^{-2}$}

The goal here is to prove the following.
\begin{prop}\label{prop:control_hms}
Let $T\in(0,1)$, $\dt>0$ small enough and $\phi(\dt)$ verifying \eqref{eq:cond_dt}. Then, there exists a continuous and linear map $L_{T}^{\sdt}:[L^2(\Omega)]^2\to L^2_{\sdmT}(0,T;L^2(\omega))$ such that for all initial data $(u_0,v_0)\in [L^2(\Omega)]^2$, there exists a time-discrete control $h=L_{T}^{\sdt}(u_0,v_0)$ such that the solution to \eqref{eq:ks_heat} satisfies
\begin{equation}\label{eq:est_tiempo_M}
\norme{(u^M,v^{M})}_{H^{-1}(\Omega)\times H^{-2}(\Omega)}\leq C_T\sqrt{\phi(\dt)}\norme{(u_0,v_0)}_{[L^2(\Omega)]^2}
\end{equation}
and
\begin{equation}
\norme{h}_{L^2_{\sdmT}(0,T;L^2(\omega))}\leq C_T\norme{(u_0,v_0)}_{[L^2(\Omega)]^2},
\end{equation}
where $C_T>0$ is the constant appearing in \Cref{eq:prop_obs}.
\end{prop}

\begin{proof} Consider the adjoint system \eqref{eq:adj_ks_heat} and the relaxed observability inequality of \Cref{eq:prop_obs}. We begin by proving the case where $\phi(\dt)={e^{-{C_1}/{(\dt)^{1/10}}}}$.
%
%\begin{equation}
%\snorme{p^{\frac12}}^2_{L^2(\Omega)}+\snorme{q^{\frac12}}^2_{L^2(\Omega)}\leq C_T\left(\ddbint_{\omega\times(0,T)}|q|^2+\phi(\dt)\left[\snorme{\difx p_T}^2_{L^2(\Omega)}+\snorme{\difx^2 q_T}^2_{L^2(\Omega)}\right]\right)
%\end{equation}

We introduce the time-discrete functional
\begin{align}\notag
J_{\dt}(p_T,q_T)=\frac{1}{2}\ddbint_{\omega\times(0,T)}|q|^2+\frac{\phi(\dt)}{2}\norme{(p_T,q_T)}_{H_0^1(\Omega)\times H^2_0(\Omega)}^2+(u_0,p^{\frac12})_{L^2(\Omega)}+(v_0,q^{\frac12})_{L^2(\Omega)}, \\ \label{eq:func_J_dt}
\quad\forall (p_T,q_T)\in H_0^1(\Omega)\times H_0^2(\Omega).
\end{align}
Clearly, the above functional is continuous and strictly convex. Since $p_T\in H_0^1(\Omega)$ (resp. $q_T\in H_0^2(\Omega)$), we have from Poincar\'e inequality that $\|p_T\|^2_{H_0^1(\Omega)}=\int_{\Omega}|\difx p_T|^2$ (resp. $\|q_T\|^2_{H_0^2(\Omega)}=\int_{\Omega}|\difx^2 q_T|^2$), and thus the observability inequality \eqref{eq:obs_final} implies the coerciveness of $J_{\dt}$. This guarantees the existence of a minimizer that we denote by $(\widehat{p_T}, \widehat{q_T})$.

Now, consider $(\widehat{p},\widehat{q})$ (resp. $(p,q)$) the solution to \eqref{eq:adj_ks_heat} with initial data $(\widehat{p_T}, \widehat{q_T})$ (resp. $(p_T,q_T)$). The Euler-Lagrange equation associated with the minimization of the functional \eqref{eq:func_J_dt} reads as
\begin{align}\notag
\ddbint_{\omega\times(0,T)}\widehat{q}q+\phi(\dt)\left[(\difx \widehat{p_T},\difx p_T)_{L^2(\Omega)}+(\difx^2 \widehat{q_T},\difx^2 q_T)_{L^2(\Omega)}\right]+ (u_0,p^{\frac12})_{L^2(\Omega)}+(v_0,q^{\frac12})_{L^2(\Omega)}=0, \\ \label{eq:E_L}
\quad\forall (p_T,q_T)\in H_0^1(\Omega)\times H_0^2(\Omega).
\end{align}
We set the control $h=L_{T}^{\sdt}(u_0,v_0)=(\chi_{\omega}\widehat{q}^{\,n+\frac12})_{n\in\inter{0,M-1}}$ and consider the solution $(u,v)$ of the controlled problem 
\begin{equation*}
\begin{cases}
\D (\Dt u)^{n+\frac12}-\Gamma \difx^2 \taup{u}^{n+\frac12}+c\difx \taup{u}^{n+\frac12}=\taup{v}^{n+\frac12} &n\in\inter{0,M-1}, \\
\D (\Dt v)^{n+\frac12}+\gamma \difx^4 \taup{v}^{n+\frac12}+\difx^3 \taup{v}^{n+\frac12}+a\difx^2 \taup{v}^{n+\frac12}=\taup{u}^{n+\frac12}+\chi_{\omega} \widehat{q}^{\,n+\frac12} &n\in\inter{0,M-1}, \\
%\left(v_{|\partial\Omega}\right)^{n+1}=\left(v_{|\partial\Omega}\right)^{n+1}=\left(\difx v_{|\partial\Omega}\right)^{n+1} =0 &n\in\inter{0,M-1}, \\
u^0=u_0, \quad v^0=v_0.
\end{cases}
\end{equation*}
Multiplying this system by $(p^{n+\frac12},q^{n+\frac12})$ at each point of $\dmT$ and integrating in $L^2(\Omega)$, we have after integration by parts and summing in $n\in\inter{0,M-1}$ that
\begin{equation}\label{eq:duality}
\ddbint_{\omega\times(0,T)} \widehat{q}q=\langle u^{M},p_T\rangle_{-1,1} + \langle v^{M},q_T\rangle_{-2,2}-(u_0,p^{\frac12})_{L^2(\Omega)}-(v_0,q^{\frac12})_{L^2(\Omega)},
\end{equation}
where $\langle\cdot,\cdot\rangle_{-s,s}$ denotes the duality pairing between $H_0^s(\Omega)$ and $H^{-s}(\Omega)$, $s\in\mathbb N^*$. With \eqref{eq:E_L} and \eqref{eq:duality}, we deduce that
\begin{equation}\label{eq:estado_final}
(u^M,v^M)=-\phi(\dt)\left((-\difx^2)\widehat{p_T},\difx^4\widehat{q_T}\right).
\end{equation}

By taking $(p_T,q_T)=(\widehat{p_T},\widehat{q_T})$ in \eqref{eq:E_L}, we get 
\begin{align*}
\norme{h}_{L^2_{\sdmT}(0,T;L^2(\omega))}^2+\phi(\dt)\norme{(\widehat{p_T},\widehat{q_T})}_{H_0^1(\Omega)\times H^2_0(\Omega)}^2&=-(u_0,\widehat{p}^{\frac12})_{L^2(\Omega)}-(v_0,\widehat{q}^{\frac12})_{L^2(\Omega)} \\
&\leq \norme{(u_0,v_0)}_{[L^2(\Omega)]^2} \snorme{(\widehat{p}^{\frac12},\widehat{q}^\frac12)}_{[L^2(\Omega)]^2},
\end{align*}
thus, from \eqref{eq:obs_final}, we have $\|h\|_{L^2_{\sdmT}(0,T;L^2(\omega))}=\norme{\widehat{q}}_{L^2_{\sdmT}(0,T;L^2(\omega))}\leq C_T \norme{(u_0,v_0)}_{[L^2(\Omega)]^2}$ and
\begin{equation}\label{eq:est_final}
\sqrt{\phi(\dt)}\norme{(\widehat{p_T},\widehat{q_T})}_{H_0^1(\Omega)\times H^2_0(\Omega)} \leq C_T \norme{(u_0,v_0)}_{[L^2(\Omega)]^2}.
\end{equation}
In this way, the linear map $L_T^{\sdt}:[L^2(\Omega)]^2\to L^2_{\sdmT}(0,T;L^2(\omega))$ is well defined and continuous. Finally, expressions \eqref{eq:estado_final} and \eqref{eq:est_final} together with the definition of the $H^{-s}$-norm, $s\in\mathbb N^*$, yield \eqref{eq:est_tiempo_M}. This ends the proof for $\phi(\dt)=e^{-C_1/(\dt)^{1/10}}$.

The case of a general function $\phi$ follows similarly. Indeed, for any given $\phi(\dt)$ verifying \eqref{eq:cond_dt}, we see that there exists some $\ov{\dt}>0$ such that $\phi(\dt)=e^{-C_1/(\dt)^{1/10}}\leq \phi(h)$ for all $0<\dt\leq \ov{\dt}$ and, decreasing $\dt$ if necessary, the observability inequality \eqref{eq:obs_final} also holds for the function $\phi(\dt)$. The rest of the proof is the same. 
\end{proof}

\subsection{Proof of \Cref{thm:main_1}: controllability in $L^2$}
The proof of \Cref{thm:main_1} relies on \Cref{prop:control_hms}. The idea is to steer the solution to a small target in $H^{-1}\times H^{-2}$ and then let the solution evolve uncontrolled. Let $\phi$ be a function satisfying \eqref{eq:cond_dt} and set $\widetilde{\phi}(\dt)=\dt\,\phi(\dt)$. Notice that $\widetilde\phi$ also satisfies \eqref{eq:cond_dt} for a (possibly) different constant $C_1>0$.

Let us fix $T\in(0,1)$, the initial data $(u_0,v_0)\in [L^2(\Omega)]^2$ and the partition \eqref{eq:discr_points}. We choose some $T_0<T$ and set $M_0=\left \lfloor{\frac{T_0}{\dt}}\right \rfloor$. From \Cref{prop:control_hms}, there exists a time-discrete control $h_0=(h_0^{n+\frac12})_{n\in\inter{0,M_0-1}}$ with $\norme{h_0}_{L^2_{\sdmT}{(0,T;L^2{\omega})}}\leq C_{T_0}\norme{(u_0,v_0)}_{[L^2(\Omega)]^2}$ and such that the solution to 
\begin{equation}\label{eq:control_h0}
\begin{cases}
\D \frac{u^{n+1}-u^{n}}{\dt}-\Gamma \difx^2 u^{n+1}+c\difx u^{n+1}={ v^{n+1}} &n\in\inter{0,M-1}, \\
\D \frac{v^{n+1}-v^{n}}{\dt}+\gamma \difx^4 v^{n+1}+\difx^3 v^{n+1}+a\difx^2 v^{n+1}={ u^{n+1}}+\chi_{\omega} h_0^{n+\frac12} &n\in\inter{0,M-1}, \\
\left(v_{|\partial\Omega}\right)^{n+1}=\left(v_{|\partial\Omega}\right)^{n+1}=\left(\difx v_{|\partial\Omega}\right)^{n+1} =0 &n\in\inter{0,M-1}, \\
u^0=u_0, \quad v^0=v_0,
\end{cases}
\end{equation}
verifies
\begin{equation}\label{eq:bound_M0}
\norme{(u^{M_0},v^{M_0})}_{H^{-1}(\Omega)\times H^{-2}(\Omega)}\leq C_{T_0}\sqrt{\widetilde \phi(\dt)}\norme{(u_0,v_0)}_{[L^2(\Omega)]^2},
\end{equation}
where $C_{T_0}$ is the observability constant corresponding to the interval $(0,T_0)$ as given in \Cref{eq:prop_obs}. This defines the state $(u^n,v^n)$ for all $n\in\inter{0,M_0}$.

Now, we set $h^{n+\frac12}\equiv 0$ for $n\in\inter{M_0,M-1}$ and consider the uncontrolled system 
\begin{equation}\label{eq:control_sin_h0}
\begin{cases}
\D \frac{u^{n+1}-u^{n}}{\dt}-\Gamma \difx^2 u^{n+1}+c\difx u^{n+1}={ v^{n+1}} &n\in\inter{M_0,M-1}, \\
\D \frac{v^{n+1}-v^{n}}{\dt}+\gamma \difx^4 v^{n+1}+\difx^3 v^{n+1}+a\difx^2 v^{n+1}={ u^{n+1}} &n\in\inter{M_0,M-1}, \\
\left(v_{|\partial\Omega}\right)^{n+1}=\left(v_{|\partial\Omega}\right)^{n+1}=\left(\difx v_{|\partial\Omega}\right)^{n+1} =0 &n\in\inter{M_0,M-1},
\end{cases}
\end{equation}
with initial data $y^{M_0}$ coming from the sequence \eqref{eq:control_h0}. Observe that for $n=M_0$, $(u^{M_0+1},v^{M_0+1})$ solves the elliptic system
\begin{equation*}
\begin{cases}
-\dt \,\Gamma \difx^2 u^{M_0+1}+\dt \,c\difx u^{M_0+1}+ u^{M_0+1}= \dt v^{M_0+1}+u^{M_0} &\text{in } (0,1), \\
\dt \gamma \difx^4 v^{M_0+1}+\dt \difx^3 v^{M_0+1}+\dt a \difx^2 v^{M_0+1}+v^{M_0+1}=\dt u^{M_0+1}+v^{M_0}  &\text{in } (0,1), \\
v^{M_0+1}=v^{M_0+1}=\difx v^{M_0+1} =0 &\text{on } \{0,1\}.
\end{cases}
\end{equation*}
In view of Lax-Milgram lemma, \eqref{eq:bound_M0} and a procedure similar to \eqref{lem:est_method}, we can deduce the regularity estimate
\begin{equation*}
\sqrt{\dt}\snorme{(u^{M_0+1},v^{M_0+1})}_{H_0^1(\Omega)\times H_0^2(\Omega)}\leq C \snorme{(u^{M_0},v^{M_0})}_{H^{-1}(\Omega)\times H^{-2}\times(\Omega)}
\end{equation*}
for all $\dt>0$ small enough and some $C>0$ only depending on $\Omega$, $\Gamma$, $\gamma$, $a$, and $c$. This, together with \eqref{eq:bound_M0} yields
\begin{equation*}
\snorme{(u^{M_0+1},v^{M_0+1})}_{H_0^1(\Omega)\times H_0^2(\Omega)} \leq C\sqrt{\frac{\widetilde{\phi}(\dt)}{\dt}}\norme{(u_0,v_0)}_{[L^2(\Omega)]^2}= C \sqrt{\phi(\dt)}\norme{(u_0,v_0)}_{[L^2(\Omega)]^2}.
\end{equation*}

Arguing as we did in \Cref{eq:prop_obs}, we can iterate for indices $n\in\inter{M_0+1,M}$ and  using Poincar\'e inequality we can deduce $\snorme{(u^{M},v^{M})}_{[L^2(\Omega)]^2}\leq C \sqrt{\phi(\dt)}\norme{(u_0,v_0)}_{[L^2(\Omega)]^2}$ for some constant depending on $T$ and $T_0$. Therefore, by means of the auxiliary problems \eqref{eq:control_h0} and \eqref{eq:control_sin_h0}, we have constructed a sequence $(u,v)=\{u^{n},v^{n}\}_{n\in\inter{0,M}}$ such that $(u^M,v^M)$ verifies a $\phi(\dt)$-null controllability constraint in $[L^2(\Omega)]^2$. This ends the proof.

\section{Concluding remarks}\label{sec:conc}
In this paper, we have studied the controllability of a time-discrete simplified stabilized Kuramoto-Sivashinsky equation. The main ingredient is a relaxed observability inequality which in turn is a consequence of time-discrete Carleman estimates with the same weight for the fourth- and second-order parabolic operators. Once this observability inequality is obtained, the controllability of the system is achieved in two steps: first we control to a small target in $H^{-1}\times H^{-2}$ and then we use some regularity results to deduce the controllability in $L^2$.

As we have mentioned in the introduction, two important simplifications have been done to obtain our results. The first one is that the system is linear. Although the controllability of time-discrete semilinear systems is in fact possible (see e.g. \cite{BLR14} in the semi-discrete case or \cite{BHS20} for the time-discrete one), the nonlinearities in such works are regarded as globally Lipschitz, which is not the case for the nonlinearity $uu_x$ in \eqref{ks_intro}. How to treat this type of nonlinearity in the discrete case remains as an open and interesting problem. 

The second simplification concerns the coupling terms in the right-hand side of \eqref{eq:ks_heat}. As compared to \eqref{ks_intro}, these couplings are easier to handle with the usual Carleman estimates and including them in our current framework require further investigation. One possible way to consider only first-order couplings is to argue as in \cite[Theorem 3.1]{CMP15}, where the authors prove a Carleman inequality for the parabolic operator $(L_{\sdmT}\difx p)$ (see definition \eqref{eq:parab_op} in our case). Such estimate can be deduced from the Carleman estimate with Fourier boundary conditions shown in \cite[Theorem 1.1]{FCGBGP06} whose proof relies on the classical duality method introduced in \cite{FI96}. A priori, it seems that this result can be deduced for the time-discrete case, but a close inspection to the proof of \cite[Theorem 1.1]{FCGBGP06} highlights some difficulties with our Carleman estimate in \Cref{lem:carleman_heat} (with the important change of Dirichlet to Neumann boundary conditions) and shows an incompatibility when considering the terms in $\mathcal W_{H}$ (particulartly the one at $n=\frac{1}{2}$). In any case, this problem deserves further attention.

\appendix

\section{Time-discrete calculus results}\label{app:discrete_things}

The goal of this section is to provide a summary of calculus rules for manipulating the time-discrete operators $\Dt$ and $\Dtbar$, and also to provide estimates for the application of such operators on the weight functions. 

To avoid introducing additional notation, we present the following continuous difference operator. For a function $f$ defined on $\R$, we set
\begin{align*}
&\tcp f(t):=f(t+\tfrac{\dt}{2}), \quad \tcm f(t):=f(t-\tfrac{\dt}{2}), \quad \Dtc f:=\frac{1}{\dt}(\tcp-\tcm)f.
\end{align*}
In this way, discrete versions of the results given below will be natural. With the notation given in the introduction, for a function $f$ continuously defined on $\R$, the discrete function $\Dt f$ amounts to evaluate $\Dtc f$ at the mesh points $\dmT$ and $\Dtbar f$ is $\Dtc f$ sampled at the mesh points $\mT$.

\subsection{Time-discrete calculus formulas}

\begin{lem}
Let the functions $f_1$ and $f_2$ be continuously defined over $\mathbb R$. We have
\begin{equation*}
\Dtc(f_1f_2)=\tcp{f_1}\,\Dtc f_2+\Dtc f_1\,\tcm{f_2}, \quad \Dtc(f_1f_2)=\tcm{f_1}\,\Dtc f_2+\Dtc f_1\,\tcp{f_2}.
\end{equation*}
From the above formulas, if $f_1=f_2=f$, we have the useful identities
\begin{align*}
&\tcp{f}\,\Dtc f=\frac{1}{2}\Dtc \left(f^2\right)+\frac{1}{2}\dt(\Dtc f)^2,\qquad \tcm{f}\,\Dtc f=\frac{1}{2}\Dtc\left(f^2\right)-\frac{1}{2}\dt(\Dtc f)^2.
\end{align*}
\end{lem}

The translation of the result to discrete functions $f,g_1,g_2\in H^{\overline{\dmT}}$ is 
\begin{equation}\label{deriv_prod}
\begin{split}
&\Dtbar(g_1g_2)=\taubp{g_1}\Dtbar g_2+\Dtbar g_1\taubm{g_2}, \quad \Dtbar(g_1g_2)=\taubm{g_1}\Dtbar g_2+\Dtbar g_1\taubp{g_2},
\end{split}
\end{equation}
and
\begin{align}\label{f_Dt_3}
&\taubp{f}\Dtbar f=\frac{1}{2}\Dtbar \left(f^2\right)+\frac{1}{2}\dt(\Dtbar f)^2, \quad \taubm{f}\Dtbar f=\frac{1}{2}\Dtbar\left(f^2\right)-\frac{1}{2}\dt(\Dtbar f)^2.
\end{align}
Of course, the above identities also hold for functions $f,g_1,g_2\in H^{\overline{\mT}}$ and their respective translation operators and difference operator $\mathtt{t^{\pm}}$ and $\Dt$. 

The following result covers discrete integration by parts and some useful related formulas.
\begin{prop}
Let $\{H,(\cdot,\cdot)_{H}\}$ be a real Hilbert space and consider $u\in H^{\overline{\mT}}$ and $v\in H^{\overline{\dmT}}$. We have the following:
\begin{align}\label{trans_doub}
&\dint_{0}^{T}\left(\tcp{u},v\right)_{H}=\int_{0}^{T}\left(u,\bar{\mathtt{t}}^{-}v\right)_{H}, \\
& \dint_{0}^T \left(\tcm u,v\right)_{H}=\dt(u^0,v^{\frac12})_H -\dt(u^M,v^{M+\frac12})_H+\int_{0}^{T}\left(u,\bar{\mathtt{t}}^{+}v\right)_H.\label{trans_doub2}
\end{align}
Moreover, combining the above identities, we have the following discrete integration by parts formula
\begin{equation}\label{int_by_parts}
\dint_{0}^{T} \left(D_t u,v\right)_{H}=-(u^0,v^{\frac12})_{H}+(u^M,v^{M+\frac12})_{H}-\int_{0}^{T}\left(\Dtbar v,u\right)_{H}.
\end{equation}
\end{prop}

\begin{rmk}

 If we consider two functions $f,g\in H^{\overline{\dmT}}$, we can combine \eqref{trans_doub} and \eqref{int_by_parts} to obtain the formula
 \begin{equation}\label{by_parts_same}
 \int_{0}^{T}\left(\Dtbar f,\bar{\mathtt{t}}^{-}g\right)_{H}=-(f^{\frac12},g^{\frac12})_{H}+(f^{M+\frac12},g^{M+\frac12})_{H}-\int_{0}^{T}\left(\bar{\mathtt{t}}^{+} f,\Dtbar g\right)_{H}.
 \end{equation}
 Analogously, for $f,g\in H^{\overline{\mT}}$, the following holds
  \begin{equation}\label{by_parts_same_dual}
 \dint_{0}^{T}\left(D_t f,\tcp{g}\right)_{H}=-(f^{0},g^{0})_{H}+(f^{M},g^{M})_{H}-\dint_{0}^{T}\left(\tcm f,D_t g\right)_{H}.
 \end{equation}
Observe that in these formulas, the integrals are taken over the same discrete points.
\end{rmk}

\subsection{Time-discrete computations related to Carleman weights}\label{app:theta_comps}

We present some lemmas related to time-discrete operations applied to the Carleman weights. The proof of these results can be found in \cite[Appendix B]{BHS20}. We recall that $r=e^{s\varphi}$ and $\rho=r^{-1}$. We highlight the dependence on $\tau$, $\delta$, $\dt$ and $\lambda$ in the following estimates.

\begin{lem}[Time-discrete derivative of the Carleman weight]\label{lem:deriv_lemma_time}
Let $T\in(0,1)$ and $\tau\geq 1$. Provided $\dt \tau (\delta^{m+1}T^{2m+1})^{-1} \leq 1$, we have
\begin{equation*}%\label{deriv_weight}
\tcm{(r)} \Dtc \rho=-\tau\,\tcm{(\theta^\prime)}\varphi +\dt\left(\frac{\tau}{\delta^{m+2}T^{2m+2}}+\frac{\tau^2}{\delta ^{2m+2}T^{4m+2}}\right)\mathcal O_{\lambda}(1).
\end{equation*}
\end{lem}

\begin{proof} The proof of this result can be done exactly as in \cite[Lemma B.4]{BHS20}. It relies on Taylor's formula at order 2 and the estimate $\max_{t\in[0,T]} \theta^{(j)}(t) \leq \frac{C_{m}}{\delta^{m+j}T^{2m+j}}$, $j=1,2$.
\end{proof}

\begin{lem}[Discrete operations on the weight $\theta$]\label{lemma_deriv_theta}
 There exists a universal constant $C_{\ell,m}>0$ uniform with respect to $\dt$, $\delta$ and $T$ such that
\begin{enumerate}[label=\upshape(\roman*),ref=\thelem(\roman*)]
\item \label{est_dt_square}
$ |\Dtc (\theta^\ell)| \leq C_{\ell,m} T\, \tcm{(\theta^{\ell+\frac{1}{m}})} + C_{\ell,m} \frac{\dt}{\delta^{m\ell+2}T^{2m\ell+2}}  , \quad \ell=1,2,\ldots$
% \\ \label{est_theta_prime}
%&|\tcp{(\theta^\prime)}|\leq T\tcm{(\theta)^2}+C\frac{\dt}{\delta^3T^4} \\  
%
%\item \label{est_dt_thp}
%$\Dtc(\theta^\prime) \leq C T^2\, \tcm{(\theta^3)}+C\frac{\dt}{\delta^4T^5}$.
\item \label{shift} 
$\tcm(\theta^\ell)\leq \tcp(\theta^\ell)+C_{\ell,m}\frac{\dt}{\delta^{m\ell+1}T^{2m\ell+1}}, \quad \ell=1,2, \ldots$
%\\ & |D_t(\theta^n)|\leq nT\taubp{\theta}^{n+1}+C\frac{\dt}{T^{2n+2}\delta^{n+2}}, \quad n=1,\ldots
\end{enumerate}
\end{lem}

\begin{proof} The proof of this result can be done exactly as in \cite[Lemma B.5]{BHS20}. From the estimate $\max_{t\in[0,T]}\partial_t^{j}(\theta^\ell)\leq \frac{C_{\ell,m,j}}{\delta^{m\ell+j}T^{2m\ell+j}}$, for $j,\ell=1,2$, inequality (i) follows for $j=2$ and Taylor's formula at order 2. Similarly, (ii) follows from such estimate at $j=1$ and Taylor's formula at order 1.
\end{proof}

\begin{rmk}\label{rmk:shifts_equiv} As it will be of interest during the proof of \Cref{prop:car_one_obs}, \Cref{lemma_deriv_theta} is also valid (for a possibly different universal constant) if we replace $\tcm$ by $\tcp$ and $\tcp$ by $\tcm$  everywhere. 
\end{rmk}

\section{Carleman estimate for the time-discrete fourth-order parabolic operator}\label{app:car_fourth}

We devote this section to present the proof of \Cref{thm:time_discrete_fourth}. For clarity, we have divided the proof in several steps and we mainly focus in those containing time-discrete computations. 

As in other works addressing Carleman estimates, we will keep track of the dependency of all the constants with respect to the parameters $\lambda$, $\tau$ and $T$. Also, in accordance with the nature of our problem, we pay special attention to the discrete parameters $\dt$ and $\delta$.

\subsection{The change of variable}
Without loss of generality, let $\gamma=1$ and assume for the time being that $q\in (C^4([0,1])\cap H_0^2(0,1))^{\ov{\sdmT}}$. Recalling \eqref{eq:def_varphi} and \eqref{eq:def_theta}, we introduce the following instrumental functions
\begin{align*}
s(t)=\tau\theta(t), \quad \tau>0, \quad t\in(-\delta T,T+\delta T), \\
r(t,x)=e^{s(t)\varphi(x)}, \quad \rho(t,x)=(r(t,x))^{-1}, \quad x\in [0,1], \quad t\in(-\delta T,T+\delta T),
\end{align*}
and set the change of variables
\begin{equation}\label{eq:change_var}
z^{n+\frac{1}{2}}=r(t_{n+\frac{1}{2}},\cdot) q^{n+\frac12}, \quad n\in\inter{0,M}.
\end{equation}
For the remainder of this section, we will simplify the notation in \eqref{eq:change_var} by simply writing $z=rq$ which implicitly means that the (continuous) weight function $r$ is evaluated on the same time grid (in this case the dual grid $\dmT$) and at the same time point as the one attached to the discrete variable. This will not lead to any ambiguity.

From the change of variables \eqref{eq:change_var} and since $\beta\in C^4([0,1])$ and $q\in C^4([0,1])^{\ov{\sdmT}}$, we can obtain --after a very long, but straightforward computation-- that
\begin{equation}\label{eq:change_dxxxx}
r \partial_x^4(\rho z)= \mathcal {LT}+ \mathcal{RT}
\end{equation}
where 
\begin{align*}\notag 
\mathcal{LT}:=&-4 s^3\lambda^3 (\beta_x)^3\phi^3  \difx z+s^4\lambda^4 (\beta_x)^4\phi^4  z+6s^2\lambda^2 (\beta_x)^2\phi^2  \difx^2 z \\
& -4 s\lambda\beta_x \phi  \difx^3 z + \difx^4 z+6 s^2 \lambda^2 \left((\beta_x)^2\phi^2\right)_x  \difx z
\end{align*}
and
\begin{align*}\notag 
\mathcal{RT}:=&-s \lambda \beta^{(4)}\phi  z-4s\lambda^2 \beta_x \beta_{xxx}\phi  z+4 s^2\lambda^2\beta_{xxx}\beta_{x}\phi^2  z-4 s\lambda \beta_{xxx}\phi  \difx z-3 s\lambda^2(\beta_{xx})^2\phi  z \\ \notag
&-6 s \lambda^3 (\beta_x)^2\beta_{xx}\phi  z+18 s^2\lambda^3 (\beta_x)^2 \beta_{xx}\phi^2  z + 3 s^2\lambda^2 (\beta_{xx})^2\phi^2  z-6\lambda^3 s^3 (\beta_x)^2\beta_{xx}\phi^3  z \\ \notag
&-6s \lambda \beta_{xx}\phi  \difx^2 z-s\lambda^4(\beta_x)^4 \phi  z+7 s^2\lambda^4 (\beta_x)^4 \phi^2  z - 6 s^3\lambda^4 (\beta_x)^4 \phi^3  z - 6 s\lambda^2 (\beta_x)^2 \phi  \difx^2 z \\
&-12 s\lambda^2 \beta_x \beta_{xx}\phi  \difx z-4 s \lambda^3 (\beta_x)^3 \phi  \difx z.
\end{align*}

On the other hand, using \eqref{deriv_prod}, we have
\begin{equation*}
-\Dtbar(\rho z)=-\taubm{\rho}\Dtbar z-\Dtbar\rho \taubp{z}.
\end{equation*}
Moreover, multiplying the above equation by $\taubm{r}$ and taking into account \Cref{lem:deriv_lemma_time}, we get
\begin{equation*}
-\taubm{r}\Dtbar(\rho z)=-\Dtbar z +\tau \varphi \taubm{\theta^\prime}\taubp{z}-\mathcal O_{\lambda}(1)\frac{\dt \tau^2}{\delta^{2m+2}T^{4m+2}}\taubp{z},
\end{equation*}
where we have used that $\tau\geq 1$ and $T\in(0,1)$ to simplify in the last term. From here, using that $\dt \Dtbar{z}=\taubp{z}-\taubm{z}$, we get
\begin{equation}\label{eq:change_Dtbar}
-\taubm{r}\Dtbar(\rho z)=-\Dtbar z + \tau \varphi\,\tbm(\theta^\prime z)+\dt \tau\varphi \taubm{\theta^\prime}\Dtbar{z}-\mathcal O_{\lambda}(1)\frac{\dt \tau^2}{\delta^{2m+2}T^{4m+2}}\taubp{z}.
\end{equation}

Since $q=\rho z$ and both $z$ and $q$ are naturally attached to the dual time grid $\dmT$, we can see from \eqref{eq:4th_order_op} that
\begin{equation}\label{eq:change_var_eq}
-\taubm{r}\Dtbar(\rho z)+\taubm{r}\difx^4\tbm(\rho z)= \taubm{r}(P_{\sdmT} q)
\end{equation}
thus, putting together \eqref{eq:change_Dtbar}, \eqref{eq:change_dxxxx} (after an application of the discrete shift $\tbm$) and  \eqref{eq:change_var_eq}, we can conveniently write the Carleman identity
\begin{equation}\label{eq:def_car_identity}
Az+Bz=g,
\end{equation}
where $Az=\sum_{i=1}^4 A_i z$, $Bz=\sum_{i=1}^{3}{B_i} z$,
\begin{align}\notag 
g&= \taubm{r}(P_{\sdmT}q)-\tbm(\mathcal RT)-\tau \varphi\,\tbm(\theta^\prime z) \\ \label{def:g_term}
&\quad - \dt \tau \varphi \taubm{\theta^\prime}\Dtbar z +\mathcal O_{\lambda}(1)\frac{\dt \tau^2}{\delta^{2m+2}T^{4m+2}}\taubp{z},
\end{align}
and with
\begin{align}\notag 
&A_1 z= 6 \tau^2 \taubm{\theta}^2 \lambda^2 (\beta_x)^2 \phi^2 \difx^2\taubm{z}, \quad A_2 z=\tau^4\taubm{\theta}^4\lambda^4 (\beta_x)^4 \phi^4 \taubm{z}, \\ \label{eq:A}
& A_3 z=\difx^4\taubm{z}, \quad A_4 z=6 \tau^2\taubm{\theta}^2 \lambda^2 \left((\beta_x)^2\phi^2\right)_x \difx\taubm{z},
\end{align}
and
\begin{align}\label{eq:B}
B_1z=-\Dtbar z, \quad B_2 z=-4\tau^3\taubm{\theta}^3\lambda^3 (\beta_x)^3 \phi^3 \difx\taubm{z}, \quad B_3z=-4 \tau \taubm{\theta}\lambda \beta_x \phi \difx^3 \taubm{z}.
\end{align}
At this point, we have used that $s=\tau\theta$.

From \eqref{eq:def_car_identity}, we readily identify the identity sought in the classical methodology developed in \cite{FI96}. The splitting we have used here is inspired by the work \cite{CMP15}. Comparing to the identity shown there, we see that we have two additional terms in the right hand side of \eqref{eq:def_car_identity} with a factor $\dt$ in front and, of course, we have the time discrete derivative $\Dtbar z$ instead of the continuous one in the term $B_1 z$. 

\subsection{Estimate of the cross product}
Notice that equality \eqref{eq:def_car_identity} is written in the primal mesh $\mT$. Thus, following the classical method, we take the $L^2_{\smT}$-norm on both sides, which yields
\begin{equation}\label{eq:carleman_eq}
\norme{Az}^2_{L^2_{\smT}(Q)}+\norme{Bz}^2_{L^2_{\smT}(Q)}+2\left(Az,Bz\right)_{L^2_{\smT}(Q)}=\norme{g}^2_{L^2_{\smT}(Q)}.
\end{equation}
In this part, we will dedicate to estimate the term $\left(Az,Bz\right)_{L^2_{\smT}(Q)}$. Developing further, we set $I_{ij}:=(A_i z,B_j z)_{L^2_{\smT}(Q)}$.

\subsubsection{Estimates that do not involve time-discrete operations}\label{sec:cont_terms}

As it can be noticed, some of the terms in the cross product $\left(Az,Bz\right)_{L^2_{\smT}(Q)}$ do not require time-discrete computations. Indeed, only $I_{11}$, $I_{21}$, $I_{31}$ and $I_{41}$ require a time-discrete integration by parts. All of the others, can be carried out with a standard integration by parts in the space variable and just bearing in mind the notation for the shift $\tbm$. To keep this step short, we only present the final result for such terms.

Employing the notation $s(t)=\tau\theta(t)$, we have
\begingroup
\allowdisplaybreaks
\begin{align*}
& \bullet\  I_{12}=60\dbint_{Q}\taubm{s}^5\lambda^5(\beta_x)^4\beta_{xx}\phi^5 |\difx\taubm{z}|^2 + 60 \dbint_{Q}\taubm{s}^5\lambda^6 (\beta_x)^6 \phi^5 |\difx\taubm{z}|^2, \\
 &\bullet\ I_{13}=-12 \left. \int_0^{T}\taubm{s}^3\lambda^3(\beta_x)^3\phi^3 |\difx^2\taubm{z}|^2 \ \right|_{0}^{1}+36 \dbint_{Q} \taubm{s}^3\lambda^3(\beta_x)^2 \beta_{xx}\phi^3 |\difx^2\taubm{z}|^2 \\
 &\qquad \qquad +36 \dbint_{Q}\taubm{s}^3\lambda^4(\beta_x)^4\phi^3|\difx^2\taubm{z}|^2, \\
 &\bullet\ I_{22}=14 \dbint_{Q} \taubm{s}^7 \lambda^7 (\beta_x)^6 \beta_{xx}\phi^7 \taubm{z}^2+14 \dbint_{Q} \taubm{s}^7 \lambda^8 (\beta_{x})^8 \phi^7 \taubm{z}^2, \\
 &\bullet \ I_{23}=-28\dbint_{Q}\taubm{s}^5\lambda^6 (\beta_x)^6\phi^5|\difx\taubm{z}|^2-30\dbint_{Q}\taubm{s}^5\lambda^5(\beta_x)^4\beta_{x}\phi^5|\difx\taubm{z}|^2 \\
 &\qquad\qquad + 2 \dbint_{Q}\taubm{s}^5\lambda^5\left(\left(\beta_x\right)^5\phi^5\right)_{xxx}\taubm{z}^2, \\
& \bullet \ I_{32}=-6\dbint_{Q} \taubm{s}^3\lambda^3 (\beta_x)^2\beta_{xx} \phi^3 |\difx^2\taubm{z}|^2-6 \dbint_{Q} \taubm{s}^3\lambda^4 (\beta_x)^4 \phi^3 |\difx^2\taubm{z}|^2 \\
& \qquad \qquad +2 \left. \int_0^{T}\taubm{s}^3\lambda^3(\beta_x)^3\phi^3 |\difx^2\taubm{z}|^2 \ \right|_{0}^{1}, \\
& \bullet \ I_{33}=  - 2 \left. \int_0^{T}\taubm{s}\lambda \beta_x \phi |\difx^3\taubm{z}|^2 \ \right|_{0}^{1} + 2 \dbint_{Q} \taubm{s}\lambda \beta_{xx}\phi |\difx^3\taubm{z}|^2 \\
& \qquad \qquad +2 \dbint_{Q} \taubm{s}\lambda^2 (\beta_x)^2 \phi |\difx^3\taubm{z}|^2, \\
& \bullet \ I_{42} = -48 \dbint_{Q}\taubm{s}^5 \lambda^5 (\beta_x)^4 \beta_{xx} \phi^5 |\difx\taubm{z}|^2-48 \dbint_{Q}\taubm{s}^5\lambda^6 (\beta_x)^6\phi^5 |\difx\taubm{z}|^2, \\
&\bullet \ I_{43}= 48 \dbint_{Q} \taubm{s}^3\lambda^4 (\beta_x)^4 \phi^3 |\difx^2\taubm{z}|^2+48 \dbint_{Q}\taubm{s}\lambda^3(\beta_x)^2\beta_{xx}\phi^3 |\difx^2\taubm{z}|^2 \\ 
&\qquad\qquad- 12 \dbint_{Q} \taubm{s}^3\lambda^3\left[(\beta_x)^2\phi \left((\beta_x)^2\phi^2\right)_x\right]_{xx} |\difx \taubm{z}|^2.
\end{align*}
\endgroup

Adding up all the boundary terms above and using \eqref{trans_doub} with  property \eqref{eq:positive_beta}, we see that 
\begin{equation*}
-10 \, \dint_{0}^{T} {s}^{3}\lambda^3\left(\beta_x\right)^3\phi^3 |\difx^2{z}|^2 \Bigg|_0^1  - 2 \, \dint_0^{T}  s \lambda \beta_x  \phi |\difx^3 z |^2 \Bigg|_0^1 \geq 0,
\end{equation*}
thus we shall drop them hereinafter. 

\subsubsection{Estimates involving time-discrete computations}\label{sec:discr_terms} 

\textit{- Estimate of $I_{11}$}. This is the most delicate estimate. Integrating by parts in space we have
\begin{align*}
I_{11}%&=-6 \dbint_{Q}\tau ^2 \taubm{\theta}^2 \lambda^2 (\beta^\prime)^2 \phi^2 \difx^2 \taubm{z}\Dtbar z \\
&= 6 \dbint_{Q} \tau^2 \taubm{\theta}^2 \lambda^2 \left( (\beta_x)^2\phi^2\right)_x \difx \taubm{z}\Dtbar z + 6 \dbint_{Q} \tau^2 \taubm{\theta}^2 \lambda^2  (\beta_x)^2\phi^2 \difx \taubm{z}\Dtbar(\difx z) \\
&=: I_{11}^{(1)}+I_{11}^{(2)}.
\end{align*}
We keep $I_{11}^{(1)}$ since it will be canceled in a subsequent step. For the second term, we have using formula \eqref{f_Dt_3} and then integrating in time that
\begin{align*}
I_{11}^{(2)}&= 3 \dbint_{Q} \tau^2 \taubm{\theta}^2 \lambda^2 (\beta_x)^2 \phi^2 \Dtbar[(\difx z)^2] - 3 \dt \dbint_{Q} \tau^2 \taubm{\theta}^2 \lambda^2 (\beta_x)^2 \phi^2 \left(\Dtbar \difx z\right)^2 \\
&= 3 \int_0^{1} \tau^2 (\theta^{M+\frac{1}{2}})^2 \lambda^2 (\beta_x)^2 \phi^2 |\difx z^{M+\frac{1}{2}}|^2 - 3 \int_0^{1} \tau^2 (\theta^{\frac{1}{2}})^2 \lambda^2 (\beta_x)^2 \phi^2 |\difx z^{\frac{1}{2}}|^2 \\
&\quad - \underbrace{\dbint_{Q} \tau^2 (\Dtbar \theta^2) \lambda^2 \phi^2 |\difx \taubp{z}|^2}_{:= \mathcal P} - \underbrace{3 \dt \dbint_{Q} \tau^2 \taubm{\theta}^2 \lambda^2 (\beta_x)^2 \phi^2 \left(\Dtbar \difx z\right)^2}_{:= \mathcal  Q}.
\end{align*}
We observe that the first two terms above have a prescribed sign. Note that the same is true for the term $\mathcal P$, nonetheless has the bad sign and cannot be dropped. 

Let us find a bound for $\mathcal P$. Using Lemma \ref{est_dt_square} with $\ell=2$ we have
\begin{align}
\left|\mathcal P\right| \leq C\dbint_{Q}\left\{ T\taubp{\theta}^{2+{1}/{m}}+\frac{\dt}{\delta^{2m+2}T^{4m+2}}\right\}\tau^2 \lambda^2 \left(\beta_x\right)^2 \phi^2 |\difx\taubp{z}|^2.
\end{align}
Using the definitions of $\tbm$ and $\tbp$, we can shift the indices in the above equation to deduce
\begin{align}\notag
|\mathcal P| & \leq C\dbint_{Q}\left\{ T\taubm{\theta}^{2+{1}/{m}}+\frac{\dt}{\delta^{2m+2}T^{4m+2}}\right\}\tau^2 \lambda^2 \left(\beta_x\right)^2 \phi^2 |\difx\taubm{z}|^2 \\ \notag
&\quad + C \int_{0}^{1} \left\{T(\theta^{M+\frac12})^{2+{1}/{m}}+\frac{\dt}{\delta^{2m+2}T^{4m+2}}\right\} \tau^2 \lambda^2 \left(\beta_x\right)^2 \phi^2 |\difx z^{M+\frac12}|^2 \\
&\quad -  C \int_{0}^{1} \left\{T(\theta^{\frac12})^{2+{1}/{m}}+\frac{\dt}{\delta^{2m+2}T^{4m+2}}\right\} \tau^2 \lambda^2 \left(\beta_x\right)^2 \phi^2 |\difx z^{\frac12}|^2. \label{eq:est_P}
\end{align}

Now, we turn our attention to $\mathcal Q$. Noting that $\Dtbar$ and $\difx$ commute, we can integrate by parts in the space variable to get
\begin{align*}
\mathcal Q&= 3\dt \dbint_{Q} \tau^2 \taubm{\theta}^2\lambda^2 (\beta_x)^2 \phi^2 \left(\Dtbar z\right) \left(\Dtbar \difx^2 z\right) \\
&\quad + 3\dt \dbint_{Q} \tau^2 \taubm{\theta}^2\lambda^2 \left( (\beta_x)^2 \phi^2\right)_{x} (\Dtbar z) (\Dtbar\difx z)=: \mathcal Q_1+\mathcal Q_2.
\end{align*}
We observe that there are not boundary terms since $q\in (H_0^2(0,1))^{\ov{\sdmT}}$ and thus $z^{n-\frac{1}{2}}=\difx z^{n-\frac{1}{2}}=0$  for $n\in\inter{1,M+1}$. Using Cauchy-Schwarz and Young inequalities, we get
\begin{equation*}
|\mathcal Q_1| \leq \vartheta \dt \dbint_{Q} (\Dtbar \difx^2 z)^2 + \frac{C}{\vartheta} \dt \dbint_{Q} \tau^4 \taubm{\theta}^4 \lambda^4 (\beta_x)^4 \phi^4 (\Dtbar z)^2
\end{equation*}
for any $0<\vartheta<1$ and some $C>0$ uniform with respect to $\dt$. For $\mathcal Q_2$, we have after integration by parts in space
\begin{equation*}
\mathcal Q_2=-\frac{3 \dt}{2} \dbint_{Q} \tau^2 \taubm{\theta}^2 \lambda^2 \left((\beta_x)^2\phi^2\right)_{xx} (\Dtbar z)^2.
\end{equation*}
Once again, we do not have boundary terms since $q\in (H_0^2(0,1))^{\ov{\sdmT}}$.

Overall, collecting the above terms, we have
\begin{align}\notag
I_{11}&\geq I_{11}^{(1)} - C \dbint_{Q} T\taubm{\theta}^{2+{1}/{m}}\tau^2 \lambda^2 \left(\beta_x\right)^2 \phi^2 |\difx\taubm{z}|^2 -  \vartheta \dt  \dbint_{Q} \left(\Dtbar \difx^2 z\right)^2
 \\
&\quad - C X_{11}- C\left(1+\frac{1}{\vartheta}\right) Y_{11}- C W_{11},  \label{eq:I11_final}
\end{align}
for any $0<\vartheta<1$ and some constant $C>0$, with
\begin{align*}
X_{11}&:= \frac{\dt}{\delta^{2m+2}T^{4m+2}} \dbint_{Q} \tau^2 \lambda^2 \left(\beta_x\right)^2 \phi^2 |\difx\taubm{z}|^2 
\\
Y_{11}&:= \dt \dbint_{Q} \tau^2 \taubm{\theta}^2 \lambda^2 \left|\left((\beta_x)^2\phi^2\right)_{xx}\right| \left(\Dtbar z\right)^2 + \dt \dbint_{Q} \tau^4 \taubm{\theta}^4 \lambda^4 (\beta_x)^4 \phi^4 (\Dtbar z)^2
\\
W_{11}&:=  \int_0^{1} \tau^2 (\theta^{\frac{1}{2}})^2 \lambda^2 (\beta_x)^2 \phi^2 |\difx z^{\frac{1}{2}}|^2 +  \int_{0}^{1} \left\{T(\theta^{M+\frac12})^{2+{1}/{m}}+\frac{\dt}{\delta^{2m+2}T^{4m+2}}\right\} \tau^2 \lambda^2 \left(\beta_x\right)^2 \phi^2 |\difx z^{M+\frac12}|^2.
\end{align*}

\textit{- Estimate of $I_{21}$}. Using formula \eqref{f_Dt_3} we have
\begin{align}\notag
I_{21}&=-\dbint_{Q}\tau^4\taubm{\theta}^4 \lambda^4 (\beta_x)^4\phi^4 \taubm{z}\Dtbar z\\ \notag
&=-\frac{1}{2}\dbint_{Q} \tau^4\taubm{\theta}^4 \lambda^4 (\beta_x)^4 \phi^4 \Dtbar(z^2)+\frac{\dt}{2}\dbint_{Q}\tau^4\taubm{\theta}^4\lambda^4 (\beta_x)^4 \phi^4(\Dtbar z)^2\\ \label{eq:est_init_I21}
&=: I_{21}^{(1)}+I_{21}^{(2)}.
\end{align}

We notice that $I_{21}^{(2)}$ is positive. On the other hand, discrete integration by parts in the first term yields
\begin{align}\notag %% eq referenciada indirectamente
I_{21}^{(1)}&=\frac{1}{2}\dbint_{Q}\Dtbar(\theta^4)\tau^4\lambda^4 (\beta_x)^4 \phi^4 \taubp{z}^2 - \frac{1}{2}\int_0^{1}\tau^4(\theta^{M+\frac{1}{2}})^{4}\lambda^4 (\beta_x)^4 \phi^4 |z^{M+\frac{1}{2}}|^2 \\ \label{eq:est_inter_I21}
&\quad + \frac{1}{2}\int_0^1 \tau^4(\theta^{\frac{1}{2}})^{4}\lambda^4 (\beta_x)^4 \phi^4 |z^{\frac{1}{2}}|^2.
\end{align}
Let us focus now on the first term of the above expression. Using Lemma \ref{est_dt_square} with $\ell=4$ we have
\begin{align}\notag
&\left| \frac{1}{2} \dbint_{Q} \Dtbar(\theta^4)\tau^4\lambda^4(\beta_x)^4\taubp{z}^2 \right|  \leq C \dbint_{Q} \left\{T \taubp{\theta}^{4+{1}/{m}} + \frac{\dt}{\delta^{4m+2}T^{8m+2}} \right\} \tau^4\lambda^4(\beta_x)^4 \phi^4 \taubp{z}^2
\end{align}
whence, shifting the indices in the right-hand side, we get
\begin{align}\notag
&\left| \frac{1}{2} \dbint_{Q} \Dtbar(\theta^4)\tau^4\lambda^4(\beta_x)^4\taubp{z}^2 \right| \\ \notag
&\quad \leq C \dbint_{Q} \left\{T \taubm{\theta}^{4+{1}/{m}} + \frac{\dt}{\delta^{4m+2}T^{8m+2}} \right\} \tau^4\lambda^4(\beta_x)^4 \phi^4 \taubm{z}^2 \\ \notag
&\qquad + C \int_0^{1} \left\{T (\theta^{M+\frac12})^{4+{1}/{m}} + \frac{\dt}{\delta^{4m+2}T^{8m+2}} \right\} \tau^4\lambda^4(\beta_x)^4 \phi^4 (z^{M+\frac12})^2 \\ \label{eq:est_f_12_1_1}
&\qquad - C \int_0^{1} \left\{T (\theta^{\frac12})^{4+{1}/{m}} + \frac{\dt}{\delta^{4m+2}T^{8m+2}} \right\} \tau^4\lambda^4(\beta_x)^4 \phi^4 (z^{\frac12})^2.
\end{align}

Combining \eqref{eq:est_init_I21}--\eqref{eq:est_f_12_1_1} and dropping the corresponding positive terms, we can bound $I_{21}$ from below as follows
\begin{equation}\label{eq:I21_final}
I_{21}\geq - C \dbint_{Q} T \taubm{\theta}^{4+{1}/{m}}   \tau^4\lambda^4 (\beta_x)^4  \phi^4 \taubm{z}^2 -C X_{21} -C W_{21}
\end{equation}
for some $C>0$ uniform with respect to $\dt$, with 
\begin{align*}
X_{21}&:= \frac{\dt}{\delta^{4m+2}T^{8m+2}}\dbint_{Q} \tau^4\lambda^4 (\beta_x)^4  \phi^4 \taubm{z}^2,  \\
W_{21}& :=  \int_0^{1} \left\{(\theta^{M+\frac12})^4+T (\theta^{M+\frac12})^{4+{1}/{m}} + \frac{\dt}{\delta^{4m+2}T^{8m+2}} \right\} \tau^4\lambda^4 \phi^4 (z^{M+\frac12})^2.
\end{align*}

\textit{- Estimate of $I_{31}$}. Since $q\in (H_0^2(0,1))^{\ov{\sdmT}}$ and thus $z^{n-\frac{1}{2}}=\difx z^{n-\frac{1}{2}}=0$  for $n\in\inter{1,M+1}$, we can integrate by parts in the space variable twice and obtain
\begin{align*}
I_{31} &= -\dbint_{Q}\difx^4\taubm{z}\Dtbar{z}  = -\dbint_{Q}\difx^2\taubm{z}\Dtbar(\difx^2 z)  \\
&= - \frac{1}{2} \dbint_{Q}\Dtbar(\difx^2 z)+\frac{\dt}{2}\dbint_{Q}\left(\Dtbar\difx^2 z\right)^2,
\end{align*}
and using discrete integration by parts in the first term we get
\begin{align}\notag 
I_{31}&=\frac{1}{2}\int_{0}^{1}\left|\difx^2 z^{\frac{1}{2}}\right|^2-\frac{1}{2}\int_{0}^{1}\left|\difx^2 z^{M+\frac{1}{2}}\right|^2+\frac{\dt}{2}\dbint_{Q}\left(\Dtbar\,\difx^2 z\right)^2 \\
&\geq \frac{\dt}{2}\dbint_{Q}\left(\Dtbar\,\difx^2 z\right)^2 - W_{31} \label{eq:I31_final}
\end{align}
with $W_{31}:=\frac{1}{2}\int_{0}^{1}|\difx^2 z^{M+\frac{1}{2}}|^2$. 

\begin{rmk}
Unlike for the time-discrete heat equation (see \cite[Eq. (2.20)]{BHS20}), we keep the positive term in \eqref{eq:I31_final} since it will help to absorb the similar term coming from \eqref{eq:I11_final}. We emphasize that in the continuous case $I_{31}\equiv 0$.
\end{rmk}

\textit{- Estimate of $I_{41}$}. We readily have
\begin{equation} \label{eq:I41_final}
I_{41}=-6 \dbint_{Q} \tau^2 \taubm{\theta}^2 \lambda^2 \left( (\beta_x)^2\phi^2\right)_x \difx \taubm{z}\Dtbar z = -I_{11}^{(1)}.
\end{equation}

Recalling the notation $s(t)=\tau\theta(t)$, we can put together estimates \eqref{eq:I11_final}, \eqref{eq:I21_final}, \eqref{eq:I31_final}, \eqref{eq:I41_final} and fix $\vartheta>0$ small enough to obtain
\begin{align}\notag
\sum_{i=1}^{4} I_{i1} &\geq c \dt \dbint_{Q} \left(\Dtbar \difx^2z\right)^2  - C \dbint_{Q} T\taubm{\theta}^{1/m}\taubm{s}^{2} \lambda^2 \left(\beta_x\right)^2 \phi^2 |\difx\taubm{z}|^2 \\ \label{eq:sum_I_i1}
&\quad - C \dbint_{Q} T\taubm{\theta}^{1/m} \taubm{s}^{4}  \lambda^4 \left(\beta_x\right)^4  \phi^4 \taubm{z}^2 - C X - C Y - C W,
\end{align}
for some constants $c,C>0$ uniform with respect to $\dt$, with $X:= X_{11}+X_{21}$,  $Y: =  Y_{11}$, and  $W= W_{11}+ W_{21} + W_{31}$. 

\begin{rmk} Estimate \eqref{eq:sum_I_i1} resembles the one obtained in the continuous setting, in fact, the second and third terms the expected ones. All of the other terms are related to the time-discrete nature of our problem. The terms collected in $X$ and $Y$ depend directly on the parameter $\dt$, while the terms in $W$ appear since our Carleman weight does not blow up as $t \to 0^+$ and $t\to T^{-}$.
\end{rmk}

\subsection{Towards the Carleman estimate}

For a set $U\subset \Omega$, we define
\begin{align} \notag
I_{U}(z;\lambda)&= \dbint_{Q} \taubm{s}^{7} \lambda^8 \phi^7 \taubm{z}^2 + \dbint_{Q} \taubm{s}^{5} \lambda^6 \phi^5 |\difx \taubm{z}|^2 \\ \label{eq:notation_I}
&\quad + \dbint_{Q} \taubm{s}^{3} \lambda^4 \phi^3 |\difx^2 \taubm{z}|^2 + \dbint_{Q} \taubm{s} \lambda^2 \phi |\difx^3 \taubm{z}|^2.
\end{align}
As usual in other proofs for Carleman estimates, using the properties of the weight $\beta$ (see \cref{eq:cond_beta2}) and collecting the terms $I_{ij}$ from Sections \ref{sec:cont_terms} and \ref{sec:discr_terms}, we have that for $\lambda\geq C$ and $\tau\geq CT^{2m-1/3}$ the following estimate holds
\begin{equation*}
\sum_{i=1}^{4}\sum_{j=1}^{3} I_{ij} \geq C I _{\Omega\setminus \omega_0}(z;\tau,\lambda)= C I_{\Omega}(z;\tau,\lambda)-C I_{\omega_0}(z;\tau,\lambda),
\end{equation*}
for some $C>0$ only depending on $\omega$. So far, from \eqref{eq:carleman_eq}  and the above estimate, we have
\begin{align}\label{eq:car_inter_1}
I_{\Omega}(z;\tau,\lambda)+\norme{Az}_{L^{2}_{\smT(Q)}}^2+\norme{Bz}_{L^2_{\smT(Q)}} \leq C\left(\norme{g}^2_{L^2_{\smT(Q)}} + I_{\omega_0}(z;\tau,\lambda)+X+Y+Z\right)
\end{align}
for any $\lambda\geq C$ and $\tau\geq CT^{2m-1/3}$. 

We will add terms containing $\Dtbar z$ and $\difx^4\taubm{z}$ in the left-hand side of the above equation. From \eqref{eq:A}, we have
\begin{align*}\notag
&\|\taubm{s}^{-1/2}\phi^{-1/2}\difx^4\taubm{z}\|^2_{L^2_{\smT}(Q)} \\
&\quad \leq C\left( \|\taubm{s}^{-1/2} Az\|^2_{L^2(Q)}+  \|\taubm{s}^{3/2}\lambda^2\phi^{3/2}\difx^2\taubm{z}\|^2_{L^2_{\smT}(Q)}\right)\\
&\qquad + C \left( \|\taubm{s}^{7/2}\lambda^4\phi^{7/2}\taubm{z}\|^2_{L^2_{\smT}(Q)} + \|\taubm{s}^{5/2}\lambda^3\phi^{5/2}\difx\taubm{z}\|^2_{L^2_{\smT}(Q)}\right),
\end{align*}
where we have used that $\left|(\left(\beta_x)^2\phi^2\right)_x\right|\leq C\lambda\phi^2$ and have taken $\lambda\geq 1$ and $\tau\geq CT^{2m}$. Similarly,  from \eqref{eq:B} we get
\begin{align*}\notag
\|\taubm{s}^{-1/2}\phi^{-1/2}\Dtbar z\|_{L^2_{\smT}(Q)}^2 & \leq C\|\taubm{s}^{-1/2}\phi^{-1/2}B z\|_{L^2_{\smT}(Q)}^2 + C \|\taubm{s}^{5/2}\lambda^3\phi^{3/2}\difx\taubm{z}\|^2_{L^2_{\smT}(Q)} \\
&\quad + C \|\taubm{s}^{1/2}\lambda\phi^{1/2}\difx^3\taubm{z}\|^2_{L^2_{\smT}(Q)}.
\end{align*}

Thus, taking into account the above estimates and increasing (if necessary) the value of $\tau$ such that $s(t)\geq 1$, estimate \eqref{eq:car_inter_1} becomes
\begin{align}\notag
\dbint_{Q}&\taubm{s}^{-1}\left[|\Dtbar z|^2+|\difx^4\taubm{z}|^2\right]+I_{\Omega}(z;\tau,\lambda) \\ \label{eq:car_inter_2}
& \leq C\left(\norme{g}^2_{L^2_{\smT(Q)}} + I_{\omega_0}(z;\tau,\lambda)+X+Y+Z\right)
\end{align}
for all $\lambda\geq C$ and $\tau\geq C(T^{2m}+T^{2m-1/3})$.

\subsection{Absorbing discrete related terms}

In this step, we will absorb the remaining terms in the right-hand side of \eqref{eq:car_inter_2} by choosing the Carleman parameters in a specific order. We begin by choosing $\lambda_0\geq 1$ large enough (only depending on $\omega_0$ and $\omega$) and set $\lambda=\lambda_0$ for the rest of the proof. %By a slight abuse of notation, we write $I_U(z;\tau)=I_U(z;\tau,1)$ where we recall \eqref{eq:notation_I}.

We begin with the following result.

\begin{lem} There exists $C_{\lambda_0}>0$ such that
\begin{align}\label{eq:est_g_total}
\norme{g}^2_{L^2_{\smT}(Q)} &\leq C_{\lambda_0} \left(\dbint_{Q} \taubm{r}^{2}|P_{\sdmT}q |^2 + \tau^{-1} I_{\Omega}(z;1)+ X_g+ Y_g+ W_g \right)
\end{align}
for all $\tau\geq \tau_0(T^{2m}+T^{2m-1/3})$ where $\tau_0$ is a positive constant only depending on $\lambda_0$ and 
\begin{align}
&X_g :=  \left(\frac{\dt \tau^2}{\delta^{2m+2}T^{4m+2}}\right)^2 \dbint_{Q}\taubm{z}^2, \quad Y_g := (\dt)^2 \tau^2 T^2 \dbint_{Q} \taubm{\theta}^{2(1+{1}/{m})}(\Dtbar z)^2, \\
& W_g:= \left(\frac{\dt \tau^2}{\delta^{2m+2}T^{4m+2}}\right)^2 \int_0^{1} |z^{M+\frac12}|^2  .
\end{align}
\end{lem}
\begin{proof}
The proof follows by successive applications of triangle and Young inequalities. Looking at the definition of $g$, see eq. \eqref{def:g_term}, the first term of \eqref{eq:est_g_total} is obvious. The second term comes from estimating the terms in $\mathcal{RT}$ in \eqref{def:g_term} by using the fact that $\beta\in C^4([0,1])$ and $\tau \geq C T^{2m}$ for some $C>0$ only depending on $\lambda_0$. The third term in \eqref{def:g_term} can be incorporated on the second one in \eqref{eq:est_g_total} by noting that $|\theta^\prime(t)|\leq C T \theta(t)^{1+\frac{1}{m}}$ for all $t\in(0,T)$, $\|\varphi\|_{C(\ov{\Omega})}=\O_{\lambda_0}(1)$ and taking $\tau\geq C T^{2m-1/3}$ for some $C>0$ only depending on $\lambda_0$. The terms in $X_g$, $Y_g$ and $W_g$ can be obtained straightforwardly. 
\end{proof}

Using inequality \eqref{eq:est_g_total} in \eqref{eq:car_inter_2} and since $\|\phi\|_\infty=\O_{\lambda_0}(1)$, we have
\begin{align}\notag
\dbint_{Q}&\taubm{s}^{-1}\left[\left(\Dtbar z\right)^2+|\difx^4\taubm{z}|^2\right] + I_\Omega(z) \\ \label{eq:car_inter_2}
&\leq C_{\lambda_0} \left( \dbint_{Q} \taubm{r}^2|P_{\sdmT}q|^2  + I_{\omega_0}(z)+ \underline{X}+\underline{Y}+\underline{W}\right)
\end{align}
for any $\tau\geq \tau_0(T^{2m}+T^{2m-1/3})$, where $\tau_0$ is a positive constant only depending on $\lambda_0$. In \eqref{eq:car_inter_2}, for $U\subset \Omega$ we have abridged 
\begin{align*}
I_{U}(z)&= \dbint_{U} \taubm{s}  |\difx^3 \taubm{z}|^2  + \dbint_{U} \taubm{s}^{3} |\difx^2 \taubm{z}|^2 \\
&\quad + \dbint_{U} \taubm{s}^{5} |\difx \taubm{z}|^2 + \dbint_{U} \taubm{s}^{7} \taubm{z}^2
\end{align*}
and we have collected similar terms so
\begin{align*}
\underline{X}&:= \frac{ \dt\tau^2}{\delta^{2m+2}T^{4m+2}} \dbint_{Q}  |\difx\taubm{z}|^2 +  \frac{ \dt \tau^4}{\delta^{4m+2}T^{8m+2}}\dbint_{Q}  \taubm{z}^2 + \left(\frac{\dt \tau^2}{\delta^{2m+2}T^{4m+2}}\right)^2 \dbint_{Q}\taubm{z}^2, \\
\underline{Y}&:=  \dt \dbint_{Q} \taubm{s}^2  \left(\Dtbar z\right)^2 + \dt \dbint_{Q} \taubm{s}^4  (\Dtbar z)^2 + (\dt)^2 \tau^2 T^2 \dbint_{Q} \taubm{\theta}^{2(1+{1}/{m})}(\Dtbar z)^2, \\
\underline{W}&:=  \int_0^{1} \tau^2 (\theta^{\frac{1}{2}})^2 |\difx z^{\frac{1}{2}}|^2 +  \int_{0}^{1} \left\{T(\theta^{M+\frac12})^{2+{1}/{m}}+\frac{\dt}{\delta^{2m+2}T^{4m+2}}\right\} \tau^2 |\difx z^{M+\frac12}|^2 \\
&\quad + \int_0^{1} \left\{(\theta^{M+\frac12})^4+T (\theta^{M+\frac12})^{4+{1}/{m}} + \frac{\dt}{\delta^{4m+2}T^{8m+2}} \right\} \tau^4  (z^{M+\frac12})^2 \\
&\quad + \frac{1}{2}\int_{0}^{1}|\difx^2 z^{M+\frac{1}{2}}|^2 + \left(\frac{\dt \tau^2}{\delta^{2m+2}T^{4m+2}}\right)^2 \int_0^{1} |z^{M+\frac12}|^2. 
\end{align*}

As in \cite{BHS20}, using the parameter $\dt$, we will estimate the terms in $\underline{X}$ and $\underline Y$. The results reads as follows.
\begin{lem}\label{lem:residual}
For any $\tau \geq 1$, there exists $\epsilon=\epsilon(\lambda_0)$ such that for
\begin{equation}\label{eq:dt_smallness}
{0<\frac{\dt \tau^5}{\delta^{10m}T^{14m}}} \leq \epsilon
\end{equation}
the following estimate holds
\begin{equation}
\underline{X}+\underline{Y}\leq \epsilon \left(\dbint_{Q}\taubm{s}^7 \taubm{z}^2+\dbint_{Q}\taubm{s}^5|\difx\taubm{z}|^2+\dbint_{Q}\taubm{s}^{-1}(\Dtbar z)^2\right).
\end{equation}
\end{lem} 
\begin{proof}
 We increase the value of $\tau_0$ so that $\tau_0\geq 1$ and $\tau\geq 1$.  In fact, notice that 
 \begin{align}\label{eq:sgeq1}
 \tau_0\leq \tau_0\left(1+\frac{1}{T^{1/3}}\right) \leq \tau \theta = s(t), \quad \mbox{for all } t\in[0,T].
 \end{align} 
 By assumption $\delta\leq 1/2$ and $0<T<1$, thus using that $m\geq 1/3$, we have
$ \frac{\dt \tau^2}{\delta^{2m+2}T^{4m+2}}+ \frac{\dt \tau^4}{\delta^{4m+2}T^{8m+2}}\leq  \frac{2 \dt \tau^4}{\delta^{10m}T^{14m}}$. 
 Hence, provided $\frac{\dt \tau^4}{\delta^{10m}T^{14m}} \leq \epsilon_1$ with $\epsilon_1>0$ small enough, we readily see that
\begin{equation*}
\underline{X} \leq 4(\epsilon_1+\epsilon_1^2)\left(\dbint_{Q}\taubm{s}^5|\difx\taubm{z}|^2+\dbint_{Q}\taubm{s}^7\taubm{z}^7\right).
\end{equation*}
On the other hand, recalling that $\max_{t\in[0,T]}\theta(t)\leq \frac{1}{\delta^{m}T^{2m}}$, we have
\begin{align*}
\underline{Y} &\leq \left(\frac{\dt \tau^3}{\delta^{3m}T^{6m}}+\frac{\dt \tau^5}{\delta^{5m}T^{10m}}+\frac{(\dt)^2\tau^3}{\delta^{3m+2}T^{6m+2}}\right)\dbint_{Q}\taubm{s}^{-1}(\Dtbar z)^2 \\
& \leq (2\epsilon_2+\epsilon_2^2) \dbint_{Q}\taubm{s}^{-1}(\Dtbar z)^2,
\end{align*}
provided $\frac{\dt \tau^5}{\delta^{5m}T^{10m}}\leq \epsilon_2$ for $\epsilon_2>0$ small enough. To conclude, it is enough to combine the above smallness conditions into \eqref{eq:dt_smallness} for some $\epsilon>0$ small enough.
\end{proof}

Using \Cref{lem:residual} with $\epsilon=1/2C_{\lambda_0}$, where $C_{\lambda_0}>0$ is the constant appearing in \eqref{eq:car_inter_2}, we can absorb all the terms in $\underline{X}$ and $\underline{Y}$, thus obtaining
\begin{align}\label{eq:car_inter_3}
\dbint_{Q}&\taubm{s}^{-1}\left[\left(\Dtbar z\right)^2+|\difx^4\taubm{z}|^2\right] + I_\Omega(z) \leq C_{\lambda_0} \left( \dbint_{Q} \taubm{r}^2|P_{\sdmT}q|^2  + I_{\omega_0}(z)+\underline{W}\right).
\end{align}

As in other discrete-Carleman works, we cannot remove the terms in $\underline{W}$, but just estimate them. The following result gives a bound.
\begin{lem}\label{lem:est_w}
Under the hypothesis of \Cref{lem:residual}, there exists $C>0$ only depending on $m$ such that
\begin{equation}\label{est:w_under}
\underline{W}\leq C(1+\epsilon) (\dt)^{-1}\left(\int_0^{1}|z^{M+\frac12}|^2
+\int_{0}^{1}|\difx z^{\frac12}|^2+|\difx z^{M+\frac12}|^2+\int_0^{1}|\difx^2 z^{M+\frac12}|^2\right).
\end{equation}
\end{lem}
\begin{proof}
Under the hypothesis of the lemma and recalling that $\delta\leq 1/2$, we have $\dt\leq (\delta T)^m/2^{m}$ and hence 
\begin{equation}\label{eq:cond_theta_plus}
\max_{t\in[0,T+\dt]} \theta(t)\leq \frac{2^{m}}{\delta^m T^{2m}}.
\end{equation} 
Inequality \eqref{est:w_under} follows from this estimate and direct computations. 
\end{proof}

Using estimate \eqref{est:w_under} in \eqref{eq:car_inter_3}, we obtain
\begin{align}\notag
\dbint_{Q}&\taubm{s}^{-1}\left[\left(\Dtbar z\right)^2+|\difx^4\taubm{z}|^2\right] + I_\Omega(z) \leq C_{\lambda_0} \left( \dbint_{Q} \taubm{r}^2|P_{\sdmT}q|^2  + I_{\omega_0}(z)\right) \\ \label{eq:car_inter_4}
&+ C_{\lambda_0}(\dt)^{-1}\left(\int_0^{1}|z^{M+\frac12}|^2
+\int_{0}^{1}|\difx z^{\frac12}|^2+|\difx z^{M+\frac12}|^2+\int_0^{1}|\difx^2 z^{M+\frac12}|^2\right),
\end{align}
for $\tau \geq \tau_0(T^{2m}+T^{2m-1/3})$ and $\dt \tau^{5} (\delta^{10m}T^{14m})^{-1}\leq \epsilon_0$. 

\subsection{Conclusion}

Once we have proved \eqref{eq:car_inter_4}, the conclusion follows from standard arguments, so we present them briefly. 

To have a Carleman estimate with only observation in $L^2$, let $\omega_1$ be an open set such that $\omega_0\Subset\omega_1\Subset \omega$ and a cut-off function $\rho\in C_c^\infty(\omega_1)$ such that $\rho=1$ in $\omega_0$. Then, integrating by parts
\begin{align*}\notag 
\dbint_{\omega_0\times(0,T)} \taubm{s} |\difx^3 \taubm{z}|^2 &\leq \dbint_{\omega_1\times(0,T)} \taubm{s} \rho |\difx^3 \taubm{z}|^2 \\ \notag
&= -\dbint_{\omega_1\times(0,T)} \taubm{s}\difx^4 \taubm{z} \difx^2\taubm{z}+\frac{1}{2}\dbint_{\omega_1\times(0,T)} \taubm{s}\rho_{xx} |\difx^2\taubm{z}|^2 \\
& \leq \gamma \dbint_{Q} \taubm{s}^{-1}|\difx^{4}\taubm{z}|^2 + C_{\gamma} \dbint_{\omega_1\times(0,T)} \taubm{s}^{3} |\difx^2\taubm{z}|^2
\end{align*}
for any $\gamma>0$. Here, we have also used that $s(t)\geq 1$ (see eq. \eqref{eq:sgeq1}) to adjust the powers of $s$ in the last term. By an analogous procedure, if $\omega_2$ is an open set such that $\omega_1\Subset \omega_2\Subset \omega$, then it is not difficult to see that
\begin{equation*}
\dbint_{\omega_1\times(0,T)} \taubm{s}^3 |\difx^2\taubm{z}|^2 \leq \gamma \dbint_{Q}\taubm{s} |\difx^3\taubm{z}|^2+ C_{\gamma} \dbint_{\omega_2\times(0,T)} \taubm{s}^5|\difx\taubm{z}|^2
\end{equation*}
and
\begin{equation*}
\dbint_{\omega_2\times(0,T)} \taubm{s}^5 |\difx\taubm{z}|^2 \leq \gamma \dbint_{Q}\taubm{s}^3 |\difx^2\taubm{z}|^2+ C_{\gamma} \dbint_{\omega \times(0,T)} \taubm{s}^7\taubm{z}^2.
\end{equation*}

Combining the above inequalities, putting them into \eqref{eq:car_inter_4} and taking $\gamma>0$ small enough yields
\begin{align}\notag
\dbint_{Q}&\taubm{s}^{-1}\left[\left(\Dtbar z\right)^2+|\difx^4\taubm{z}|^2\right] + I_\Omega(z) \leq C_{\lambda_0}  \dbint_{Q} \taubm{r}^2|P_{\sdmT}q|^2  \\ \label{eq:car_inter_5}
&+ C_{\lambda_0}(\dt)^{-1}\left(\int_0^{1}|z^{M+\frac12}|^2
+\int_{0}^{1}|\difx z^{\frac12}|^2+|\difx z^{M+\frac12}|^2+\int_0^{1}|\difx^2 z^{M+\frac12}|^2\right),
\end{align}
for $\tau \geq \tau_0(T^{2m}+T^{2m-1/3})$ and $\dt \tau^{5} (\delta^{10m}T^{14m})^{-1}\leq \epsilon_0$. 

Finally, we shall come back to the original variable. Recalling identity \eqref{eq:change_var}, a straightforward computation gives
\begin{align}\notag 
\frac{1}{C} r^2 & \left(s^7 |q|^2+s^5|\difx q|^2+s^3|\difx^2q|^2+s|\difx^3 q|^2 + s^{-1}|\difx^4 q|^2\right)  \\
& \leq s^7 |z|^2+s^5|\difx z|^2+s^3|\difx^2z|^2+s|\difx^3 z|^2 + s^{-1}|\difx^4 z|^2,
\end{align}
for some $C>0$ only depending on $\Omega$ and $\omega$. This estimate help us to comeback to the original variable in the terms $I_{\Omega}(z)$ of \eqref{eq:car_inter_5} and the one containing $\difx^4$. To add the term corresponding to $\Dtbar q$, we just have to notice that $-\Dtbar q= P_{\sdmT} q - \difx^4\taubm{q}$, hence
\begin{equation}
\dbint_{Q} \tbm(r^2 s^{-1}) |\Dtbar q|^2 \leq 2 \dbint_{Q} \tbm(r^2s^{-1})|P_{\sdmT}q|^2 + 2 \dbint_{Q} \tbm(r^2s^{-1})|\difx^4\taubm{q}|^2.
\end{equation}

Lastly, the terms evaluated at $n=\frac{1}{2}$ and $n={M}+\frac12$ can be estimated by noting that $\difx z= \difx r q+ r\difx q$ and $\difx^2 z= r\difx^2 q+2\difx r\difx q + \difx^2 r q$ and arguing similar to the proof of \Cref{lem:est_w}. We skip the details. This ends the proof.

\section{An auxiliary lemma}

\begin{lem}\label{lem:est_method} Let $\dt>0$ such that $2C\dt<1$ for some constant $C>0$ only depending on $\Omega$, $\Gamma$, $\gamma$, $a$, and $c$. Then, we have
\begin{align} \label{eq:est_normal}
\bullet \;\; &\snorme{p^{n-\frac12}}^{2}_{L^2(\Omega)}+\snorme{q^{n-\frac12}}_{L^2(\Omega)}^2 \leq e^{2C\dt} \left(\snorme{p^{n+\frac12}}^{2}_{L^2(\Omega)}+\snorme{q^{n+\frac12}}_{L^2(\Omega)}^2\right), \\ \notag
\bullet \;\; &\snorme{\difx p^{n-\frac12}}^{2}_{L^2(\Omega)}+\snorme{\difx^2 q^{n-\frac12}}_{L^2(\Omega)}^2 \\ \label{eq:est_regular}
&\quad \leq e^{2C\dt} \left(\snorme{\difx p^{n+\frac12}}^{2}_{L^2(\Omega)}+\snorme{\difx^2 q^{n+\frac12}}_{L^2(\Omega)}^2\right)+C\dt\, e^{2C\dt} \left(\snorme{p^{n+\frac12}}^{2}_{L^2(\Omega)}+\snorme{q^{n+\frac12}}_{L^2(\Omega)}^2\right),
\end{align}
 for all $n\in\inter{1,M}$. 

\end{lem}
\begin{proof}
Let us fix $n\in\inter{1,M}$. Recalling the identity $(a-b)a=\frac12a^2-\frac12b^2+\frac12(a-b)^2$, multiplying by $p^{n-\frac12}$ in the first equation of \eqref{eq:adj_ks_heat} and integrating by parts, we have that
\begin{align*}
\snorme{p^{n-\frac12}}^2_{L^2(\Omega)}-\snorme{p^{n+\frac12}}^2_{L^2(\Omega)}+&\snorme{p^{n-\frac12}-p^{n+\frac12}}^2_{L^2(\Omega)}+2\Gamma \dt\int_\Omega |\difx p|^2 \\
&= 2 \dt \int_{\Omega}q^{n-\frac12}p^{n-\frac12}+ 2c \dt \int_{\Omega} \difx p^{n-\frac12}p^{n-\frac12}
\end{align*}
whence
\begin{align}\label{eq:est_p_n}
\snorme{p^{n-\frac12}}^2_{L^2(\Omega)}+\Gamma \dt \int_{\Omega}|\difx p^{n-\frac12}|^2 \leq \|p^{n+\frac12}\|^2_{L^2(\Omega)} + C\dt \left(  \|q^{n-\frac12}\|^2_{L^2(\Omega)} + \|p^{n-\frac12}\|^2_{L^2(\Omega)} \right).
\end{align}

Similarly, multiplying by $q^{n-\frac12}$ in the second equation of \eqref{eq:adj_ks_heat} and integrating by parts yields
\begin{align*}
\snorme{q^{n-\frac12}}^2_{L^2(\Omega)}&-\snorme{q^{n+\frac12}}^2_{L^2(\Omega)}+\snorme{q^{n-\frac12}-q^{n+\frac12}}^2_{L^2(\Omega)}+2\gamma \dt\int_{\Omega}|\difx^2q^{n-\frac12}|^2 \\
&= 2\dt \int_{\Omega} \difx^2 q^{n-\frac12}\difx q^{n-\frac12}-2a\dt \int_{\Omega}\difx^2 q^{n-\frac12} q^{n-\frac12}+2\dt\int_{\Omega}p^{n-\frac12}q^{n-\frac12}.
\end{align*}
Thus, using that $\int_{\Omega}|u_x|^2\leq \epsilon \int_{\Omega}|u_{xx}|^2+C_\epsilon \int_{\Omega} |u|^2$ for all $\epsilon>0$ (which follows from Ehrling's lemma), together with Cauchy-Schwarz and Young inequalities, we get
\begin{equation}\label{eq:est_q_n}
\|q^{n-\frac12}\|^2_{L^2(\Omega)}+\gamma \dt \int_{\Omega}|\difx^2 q^{n-\frac12}|^2 \leq \|q^{n+\frac12}\|^2_{L^2(\Omega)}+C\dt \left( \|q^{n-\frac12}\|^2_{L^2(\Omega)}+\|p^{n-\frac12}\|^2_{L^2(\Omega)}\right)
\end{equation}
for some $C>$ only depending on $\Omega$, $\gamma$ and $a$.

Now, define $E^{n-\frac12}:= \snorme{p^{n-\frac12}}^2_{L^2(\Omega)}+\snorme{q^{n-\frac12}}^2_{L^2(\Omega)}$. Then adding estimates \eqref{eq:est_p_n} and \eqref{eq:est_q_n}, we have $(1-\dt C)E^{n-\frac12} \leq E^{n+\frac12}$ for some $C>0$ only  depending on $\Omega$, $\Gamma$, $\gamma$, $a$, and $c$. Using the inequality $e^{2x}>1/(1-x)$ for $0<x<1/2$ yields 
$E^{n-\frac12}\leq e^{2C\dt}E^{n+\frac12}$, $n\in\inter{1,M}$, which is exactly \eqref{eq:est_normal}.

To prove \eqref{eq:est_regular}, we multiply by $\difx p^{n-\frac12}$ in the first equation of \eqref{eq:adj_ks_heat} and arguing as before, we have
\begin{equation}\label{eq:est_reg_p}
\snorme{\difx p^{n-\frac12}}^2_{L^2(\Omega)}+\Gamma \dt \int_{\Omega}|\difx^2 p^{n-\frac12}|^2 \leq \snorme{\difx p^{n+\frac12}}_{L^2(\Omega)}^2 + C \dt\left(\snorme{\difx p^{n-\frac12}}^2_{L^2(Q)}+\snorme{q^{n-\frac12}}^2_{L^2(Q)}\right).
\end{equation}
In the same way, multiplying the second equation of \eqref{eq:adj_ks_heat} by $\difx^4 q^{n-\frac12}$ and using that $\int_{\Omega}|u_{xxx}|^2\leq \epsilon \int_{\Omega}|u_{xxxx}|^2+C_\epsilon \int_{\Omega} |u|^2$ for all $\epsilon>0$, we get 
\begin{align}\notag
&\snorme{\difx^2 q^{n-\frac12}}^2_{L^2(\Omega)}+\gamma \dt \int_{\Omega}|\difx^4q^{n-\frac12}|^2 \\ \label{eq:est_reg_q}
&\quad \leq \snorme{\difx^2 q^{n+\frac12}}^2_{L^2(\Omega)} + C\dt \left(\snorme{\difx^2 q^{n-\frac12}}_{L^2(\Omega)}^2+\snorme{q^{n-\frac12}}_{L^2(\Omega)}^2+\snorme{p^{n-\frac12}}^2_{L^2(\Omega)}\right).
\end{align}
Estimate \eqref{eq:est_regular} follows from \eqref{eq:est_normal}, \eqref{eq:est_reg_p}, and \eqref{eq:est_reg_q}. This ends the proof.
\end{proof}

\renewcommand{\abstractname}{Acknowledgements}
\begin{abstract}
\end{abstract}
\vspace{-0.5cm}
The author would like to thank Prof. Alberto Mercado (Universidad T\'ecnica Federico Santa Mar\'ia) for some clarifying discussions about the SKS system. 

This work has received support from the program ``Estancias posdoctorales por M\'exico'' of CONACyT, Mexico.

\bibliographystyle{alpha}
{\small
\bibliography{bib_disc_KS}
}

\end{document}